\documentclass[10pt,reqno]{amsart}
\usepackage{amsmath}
\usepackage{amssymb}
\usepackage{amsthm}
\usepackage{color}
\usepackage{graphicx}
\usepackage{psfrag}
\newcommand{\geqs}{\geqslant}
\newcommand{\leqs}{\leqslant}
\newcommand{\Nb}{{\mathbb N}}

\newcommand{\ve}{\varepsilon}
\newcommand{\sgn}{\operatorname{sgn}} 

\newtheorem{theorem}{Theorem}[section]
\newtheorem{lemma}[theorem]{Lemma}
\newtheorem{prop}[theorem]{Proposition}
\newtheorem{definition}[theorem]{Definition}
\newtheorem{coro}[theorem]{Corollary}
\newtheorem{remark}[theorem]{Remark}

\theoremstyle{definition}
\theoremstyle{remark}

\begin{document}
	
\title[Existence, uniqueness, and basic properties]{modelling silicosis: existence, uniqueness and basic properties of solutions}

\author{Fernando P. da Costa}

\address[F.P. da Costa]{Univ. Aberta, Dep. of Sciences and Technology,
  Rua da Escola Polit\'ecnica 141-7, P-1269-001 Lisboa, Portugal, and
  Univ. Lisboa, Instituto Superior T\'ecnico, Centre for Mathematical
  Analysis, Geometry and Dynamical Systems, Av. Rovisco Pais,
  P-1049-001 Lisboa, Portugal.}  \email{fcosta@uab.pt}

\author{Jo\~ao T. Pinto} 

\address[J.T. Pinto]{Univ. Lisboa, Instituto Superior T\'ecnico, Dep. of Mathematics
and Centre for Mathematical
  Analysis, Geometry and Dynamical Systems, Av. Rovisco Pais,
  P-1049-001 Lisboa, Portugal.}  \email{joao.teixeira.pinto@tecnico.ulisboa.pt}

\author{Rafael Sasportes} \address[R. Sasportes]{Univ. Aberta, Dep. of Sciences and Technology,
  Rua da Escola Polit\'ecnica 141-7, P-1269-001 Lisboa, Portugal, and
  Univ. Lisboa, Instituto Superior T\'ecnico, Centre for Mathematical
  Analysis, Geometry and Dynamical Systems, Av. Rovisco Pais,
  P-1049-001 Lisboa, Portugal.}  \email{rafael.sasportes@uab.pt}

\thanks{Research partially supported by Funda\c{c}\~ao para a Ci\^encia e a Tecnologia (Portugal) 
through project CAMGSD UIDB/MAT/04459/2020.}
\thanks{\emph{Corresponding author:} F.P. da Costa}

\date{\today}

\subjclass{Primary 34A12, 34A34; Secondary 92C50}

\keywords{Coagulation--fragmentation--death equations, model of silicosis}

\begin{abstract}
We present a model for the silicosis disease mechanism following the original proposal by Tran, Jones, and Donaldson (1995) \cite{tran}, as
modified recently by da Costa, Drmota, and Grinfeld (2020) \cite{cdg}. The model
consists in an infinite ordinary differential equation system of coagulation-fragmentation-death type.
Results of existence, uniqueness, continuous dependence on the initial data and differentiability of solutions are proved for 
the initial value problem.
\end{abstract}

\maketitle
\section{Introduction}	
Silicosis is an occupational lung disease characterized by inflammation and fibrosis of the
lungs caused by the inhalation of crystalline silica dust (also known as quartz or silicon dioxide dust,
with chemical composition {\sf SiO}$_2$) 
and can progress to respiratory failure and death \cite{lancet}. Although it is a preventable disease, it is presently incurable and affects 
hundreds of thousands of persons, killing an estimated eleven thousand annually (see figures for 2017 in \cite{GBD2017}).

\medskip

In this paper we consider a mathematical model for the process used by the immune system to clear inhalated silica
particles from the lungs' alveoli. The model is a slight modification proposed in \cite{cdg} of one introduced in \cite{tran} 
that, to the best of our knowledge, has not 
received the attention we think it deserves.

\medskip

We start by giving a simplified description of the physiological processes, essentially using \cite{tran}, followed by the 
presentation of the mathematical model to be studied.

\medskip

When a silica particle enters the respiratory system down to the pulmonary alveoli (the innermost part of the lungs) its 
presence can be detected by cells of the immune system called alveoli macrophages, either by random encounters or by chemotaxis
generated by the silica particles. Macrophages can then engulf the silica particle in a process known as phagocytosis. 
Typically, the phagocytosis of organic materials (pathogens, dying or dead cells and cellular debris) result in its 
destruction by the macrophage with the production of basic molecular fragments followed by their expulsion or assimilation  \cite{abbas}. 
In the case of silica particles, however, there is essentially nothing to be destroyed from a molecular point of view and the
role of the macrophages is just to expel their load of silica by travelling up the mucociliary escalator to get
eliminated out of the respiratory system \cite{xu}. However, due to a mechanism not yet completely understood,
silica particles are toxic to the macrophages \cite{gilberti} and lead to their 
death, which can occur while they are still in the lungs, thus liberating the silica particles back into the respiratory system.

\medskip

A macrophage can ingest more than one silica particle. The hability of a given macrophage already containing $i$ particles
to ingest an additional particle typically decreases with $i$ and in the model in \cite{tran} a maximum load capacity of 
$n_\text{max}<\infty$
is assumed \textit{a priori.} Given the toxicity of the silica particles, macrophages with a higher load of silica particles die 
at an higher rate. Additionally, the capacity of the macrophages to migrate up the mucocilliary escalator is also impaired by
an increase in their load of silica particles. It is the balance of these three types of processes that leads to the
mathematical model in \cite{tran} and that we also consider here with the modifications introduced in \cite{cdg}.

\medskip

Let us denote by $M_i = M_i(t)$ the concentration of macrophages containing $i$
silica particles (which we will call the $i$-th cohort) at time $t$, by $x=x(t)$ the
concentration of silica particles, and by $r$ (which, in principle, can be a function of $x$) the
rate of supply of new (empty of silica particles) macrophages. Following \cite{tran}, we obtain the
following equations for the mechanism described above:
\begin{eqnarray*}
    \frac{dM_0}{dt} & ~=~ & r - k_0 x M_0 - (p_0+q_0) M_0,\\
    \frac{dM_i}{dt} & ~=~ &k_{i-1} x M_{i-1} - k_i x M_i -(p_i+q_i)
    M_i, \;\; i \geqslant 1,
\end{eqnarray*}
where  $k_i$ is the rate of phagocytosis of a silica particle by a macrophage already containing
$i$ particles, $p_i$, is the transfer rate of macrophages in
the $i$-th cohort to the mucocilliary escalator, i.e. the rate of
their removal from the pulmonary alveoli together with their quartz load, and $q_i$ is the
rate of death of the macrophages in the $i$-th cohort which results in the release of the
quartz burden back into the lungs.   Unlike in \cite{tran} we do not impose an upper limit on  the
number of particles $n_\text{max}<\infty$ a macrophage can contain, the existence of such a
load capacity will be a consequence of the assumptions on the rate coefficients $k_i$ and $q_i$. 
In \cite{cdg} the following  rate equation for the evolution of the concentration of
silica particles in the system under the assumption of an inhalation rate $\alpha$ was
postulated, valid under the same assumption about the validity of the
mass action law used to obtain the equations for the $M_i$:
\begin{eqnarray*}
  \frac{dx}{dt} = \alpha - x \sum_{i=0}^\infty  k_i M_i + \sum_{i=0}^\infty q_i i M_i.
\end{eqnarray*}

A scheme of the processes modelled by the above rate equations is  presented in Figure~\ref{fig1}.

%
%
%
\begin{figure}[h]
	\includegraphics[scale=1.0]{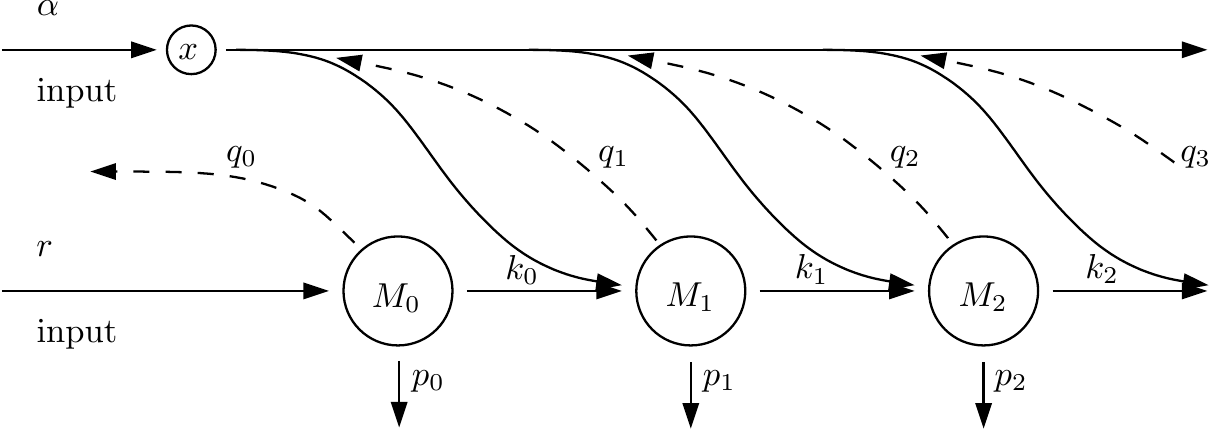}
	\caption{Reaction scheme of the mechanism studied in this paper. The input
		rates of quartz and of macrophages with no quartz
		particles are $\alpha$ and $r$, respectively. The concentration of free
		quartz particles and of macrophages containing $j$ quartz
		particles are represented by $x$ and $M_j$, respectively. 
		Macrophages $M_j$ can be destroyed, releasing $j$ quartz particles (dashed lines),
		or they can be removed by rising in the mucocilliary escalator (vertical downward
		lines), or they can ingest an additional quartz particle
		becoming an $M_{j+1}$ macrophage (horizontal rightward arrows).}\label{fig1}
\end{figure}
%
%
%

\medskip

In \cite{cdg} the equilibria of this system (with a countable infinite number) of equations was studied. An important 
modelling goal is, of course, the study of the dynamical stability properties of those equilibria and for this we need to
consider solutions which are not equilibria and to study their behaviour as $t\to +\infty$. Thus, the existence of 
solutions of the initial value problem for this system with a countable infinite number of ordinary differential equations is a problem that has to
be addressed first. This is the goal of the present paper.

\medskip

In section 2 we state, in precise mathematical terms, the problem to be studied and the basic approach we shall use, and
relate this system to others studied in the literature of cluster growth modelling. In section 3 we consider a finite dimension
truncated system that shall be used, in section 4,  to prove the existence of solutions to the Cauchy problem, which can be
considered the main result of the paper. To prove differentiability of the solutions, in section 6, as well as uniqueness, in section 7, and
the semigroup property, in section 8, a number of  \emph{a priori} estimates and evolution equations for the solution moments are needed
and will be stated and proved in section 5.

\section{Preliminaries}

The Cauchy problem we will consider in this paper is the following system of a countable number of
ordinary differential equations with nonnegative initial conditions
\begin{equation}\label{silsys}
\left\{
\begin{aligned}
&\dot{M}_0=r-k_0xM_0-(p_0+q_0)M_0\;,\\
&\dot{M}_i=k_{i-1}xM_{i-1}-k_ixM_i-(p_i+q_i)M_i\;, \quad i\geqslant 1\,,\\
&\dot{x}=\alpha -x\sum_{i=0}^\infty k_iM_i+\sum_{i=0}^\infty iq_iM_i\,,\\
&x(0)=x_0,\quad M_i(0)=M_{0i},\quad i=0,1,2,\dots\,.\rule{0mm}{6mm}
\end{aligned}
\right.
\end{equation}

This is an infinite dimensional system of ordinary differential equations of the type
coagulation-fragmentation-death type and its analysis requires the definition of an  Banach space to work in, as well as an 
adequate definition of solution. From the biological interpretation of the phase variables $x$ and $M_i$ ($i=0, 1, \ldots$)
given in the Introduction we can define the following quantities:
\begin{description}
\item[\rm Amount of quartz particles] $\displaystyle{\mathcal{X}:=x+\sum_{i=0}^\infty iM_i}$,
\item[\rm Amount of macrophages] $\displaystyle{\mathcal{M}:=\sum_{i=0}^\infty M_i}$,
\item[\rm Total amount of matter] $\displaystyle{\mathcal{U}:=\mathcal{X}+\mathcal{M}=x+\sum_{i=0}^\infty (i+1)M_i}$\,.
\end{description}
It is easy to conclude from \eqref{silsys} that, at a formal level,
\[
\dot{\mathcal{M}}= \sum_{i=0}^\infty \dot{M}_i=\dot{M}_0+\sum_{i=1}^\infty \dot{M}_i =  r-\sum_{i=0}^\infty (p_i+q_i)M_i\,,
\]
\[
\dot{\mathcal{X}}=\dot{x}+ \sum_{i=1}^\infty i\dot{M}_i = \alpha -\sum_{i=1}^\infty ip_iM_i\,,
\]
and 
\[
\dot{\mathcal{U}} = \dot{\mathcal{X}} + \dot{\mathcal{M}} = (r+\alpha) - \sum_{i=1}^\infty \bigl((i+1)p_i+q_i\bigr)M_i\,.
\]
\begin{remark}
These formal computations are justified if, for all $i=0,1,2,\dots$, and all $t\geqslant 0$, we have, 
 $M_i\geqslant 0$,
$\sum_{i=0}^\infty M_i<\infty$, 
and $\sum_{i=0}^\infty (k_i+p_i+q_i)iM_i< \infty,$ 
being these convergences uniform on each interval $[0,T]$.
\end{remark}

\noindent
Therefore, at least at a formal level, we have the following \emph{a priori} bounds: 
$\dot{\mathcal{M}}\leqslant r$, $\dot{\mathcal{X}}\leqslant\alpha$, and $\dot{\mathcal{U}}\leqslant r+\alpha$\,. 

\medskip

These formal computations and a comparison with the case of Becker-D\"oring equations \cite{BCP} and with other more general
coagulation-fragmentation systems \cite{BLL} suggest that considering a space of sequences $(x, M_0, M_1, \ldots)$ where, for nonnegative 
sequences, the quantity $\mathcal{U}$ is finite can be expected to be a mathematically reasonable space to work in and, furthermore, will 
also have biological meaning.

\medskip

Thus, let us denote an element of $\mathbb{R}^\mathbb{N}$ by $y=(y_n)=(x,M_0,M_1,\dots)$. Define
$$
\|y\|=|x|+\sum_{i=0}^\infty (i+1)|M_i|=|x|+|((i+1)M_i)_{i=0,1,\dots}|_{\ell^1}\,,
$$
and
$
X=\{y=(y_n): \|y\|<\infty\}\,.
$
We say that $y\geqslant 0$ if and only if $y\in(\mathbb{R}_0^+)^\mathbb{N}$ and will denote the nonnegative cone of $X$ by $X_+:=\{y\in X :  y\geqslant 0\}$.

\medskip

Similarly, comparison with other types of coagulation-fragmentation equations suggests the adoption of a definition of solution that is very close 
to that used in \cite{BCP} for the Becker-D\"oring system. This is not surprising, since an inspection of the model to be considered shows that, as in the 
Becker-D\"oring system, there is a phase variable that has a crucially distinct role in the dynamics: in the Becker-D\"oring system was the concentration
of monomers \cite{BCP}, in system \eqref{silsys} it is the concentration of crystaline silica, $x$. This does not mean that \eqref{silsys} can be 
automatically transformed into the Becker-D\"oring system, but it suggests analogies in the definition of solution and in the
general approch in the proofs we shall present.

\begin{definition}\label{defsol}
	Let $0<T\leqslant +\infty$. A solution of \eqref{silsys} is a function $y:[0,T)\to X$ such that:
	\begin{enumerate}
		\item[\text{(i)}] $\forall t\in[0,T)\quad y(t)\geqslant 0$;
		\item[\text{(ii)}] each $y_n:[0,T)\to \mathbb{R}$ is continuous and $\displaystyle{\sup_{t\in[0,T']}\|y(t)\|<\infty}$, for each $T'\in (0,T)$;
		\item[\text{(iii)}] $\displaystyle{\int_0^t x(s)\,ds<\infty}$, $\displaystyle{\int_0^t\sum_{i=0}^\infty (ip_i+iq_i+k_i)M_i(s)\,ds<\infty}$, for all $t\in[0,T)$;
		\item[\text{(iv)}] for all $t\in[0,T)$,
	\begin{align}
	M_0(t)&=M_{00}+rt-\int_0^t\left[ k_0x(s)M_0(s)-(p_0+q_0)M_0(s)\right]ds\,,\label{M0t}\\
	M_i(t)&=M_{0i}\nonumber\\&+\int_0^t\left[x(s)k_{i-1}M_{i-1}(s)-x(s)k_iM_i(s)-(p_i+q_i)M_i(s)\right]ds\,,\quad i\geqslant 1\,,\label{Mit}\\
	x(t)&=x_0+\alpha t+\int_0^t\Bigl[-x(s)\sum_{i=0}^\infty k_iM_i(s)+\sum_{i=0}^\infty iq_iM_i(s)\Bigr]ds\,.\label{xt}
	\end{align}
	\end{enumerate}
\end{definition}
\medskip

A standard approach used in many studies of coagulation-fragmentation type equations is to consider a
finite $n$-dimensional truncation of the infinite dimensional system and prove \emph{a priori} bounds for the evolution
of quantities associated with the truncated system that are independent of the truncation dimension $n$. This allows us to
take limits as $n\to\infty$ in the truncated system and, in this way, to prove results for the infinite system. In order to 
apply the same strategy here, we will present, in the next section, a truncated system and a useful auxiliary time evolution result.

\section{The truncated silicosis system}

In \eqref{silsys} consider rate coefficients satisfying  $k_i=0,$ for $i\geqslant n,$ and $p_i=q_i=0,$ for $i\geqslant n+1$\,.
It is clear from the biological meaning of these constants (and also from the kinetic scheme in Figure~\ref{fig1}) that,
with these conditions, if we take initial data satisfying $M_i(0)=0$ if $i\geqslant n +1$ then, for these $i$, $M_i(t)$ will remain zero for all $t>0.$
It is due to this invariance that the above conditions give rise to a finite $(n+2)$-dimensional approximation of the infinite dimensional system
\eqref{silsys}. The corresponding Cauchy problem, which we shall use extensively in the next section, can be explicitly written as follows:

\begin{equation}\label{siltrunc}
\left\{
\begin{aligned}
&\dot{M}^n_0=r-k_0x^nM^n_0-(p_0+q_0)M^n_0\;,\\
&\dot{M}^n_i=k_{i-1}x^nM^n_{i-1}-k_ix^nM^n_i-(p_i+q_i)M^n_i\;, \quad i=1,\dots,n-1\;,\\
&\dot{M}^n_n=k_{n-1}x^nM^n_{n-1}-(p_n+q_n)M^n_n\;,\\
&\dot{x}^n=\alpha -x^n\sum_{i=0}^{n-1} k_iM^n_i+\sum_{i=0}^n iq_iM^n_i\,,\\
&x^n(0)=x_0,\quad M_i^n(0)=M_{0i},\quad i=0,1,\dots, n\,.\rule{0mm}{6mm}
\end{aligned}
\right.
\end{equation}

\medskip

For the truncated versions of $\mathcal{M}$ and $\mathcal{X}$, $\mathcal{M}^n$ and $\mathcal{X}^n$ say, we have:
\begin{align*}
\dot{\mathcal{M}}^n&=\sum_{i=0}^n \dot{M}^n_i=\dot{M}^n_0+\sum_{i=1}^{n-1} \dot{M}^n_i+\dot{M}^n_n
= r-\sum_{i=0}^n (p_i+q_i)M^n_i\,,\\
\dot{\mathcal{X}}^n&=\dot{x}^n+\sum_{i=1}^{n-1} i\dot{M}^n_i+n\dot{M}^n_ n
= \alpha -\sum_{i=1}^n ip_iM^n_i\,.
\end{align*}
Therefore,
$$
\dot{\mathcal{U}}^n=r+\alpha-\sum_{i=0}^n (p_i+q_i)M^n_i -\sum_{i=1}^n ip_iM^n_i\,.
$$
For now, let's only assume that the coefficients $k_i$, $p_i$ and $q_i$ in \eqref{siltrunc}, as well as the constants $r$ and $\alpha$, are nonnegative.

\medskip

Since \eqref{siltrunc} is an initial value problem for a finite dimensional ODE with a smooth vector field, 
the standard Picard-Lindel\"{o}ff theorem allows us to conclude that, for each initial condition
$y^n_0 = (x_0^n, M_0^n, M_1^n, \ldots, M_0^n)\in\mathbb{R}^{n+2}$, there is a unique local solution in the classical sense, 
$y^n_i(\cdot):=(x^n(\cdot), M^n_0(\cdot), M^n_1(\cdot),\dots,M^n_n(\cdot))$,
such that $y^n(0)=y^n_0$. Moreover, we have:


\begin{lemma}\label{truncexist}
Let $n\in\mathbb{N}_2$. For each $y^n_0\in(\mathbb{R}_0^+)^{n+2}$, the unique local 
solution of \eqref{siltrunc} satisfying $y^n(0)=y^n_0$, 
$y^n(\cdot)$,  is extendable to the interval $[0,+\infty)$, and for each $t\geqslant 0,$ 
$y^n(t)\in (\mathbb{R}_0^+)^{n+2}.$ This solution satisfies, for all $t\geqslant 0,$
\begin{equation}\label{ynorm1}
\dot{\mathcal{U}}^n=r+\alpha-\sum_{i=0}^n (p_i+q_i)M^n_i -\sum_{i=1}^n ip_iM^n_i\leqslant r+\alpha\,.
\end{equation}
Furthermore, assume that $(g_i)\in\mathbb{R}^{n+2}.$ Then, for each positive integer $m<n$,
\begin{multline}\label{moments}
\sum_{i=m}^n g_i\dfrac{dM^n_i}{dt}+\sum_{i=m}^n g_i(p_i+q_i)M^n_i=\\ 
= g_mx^nk_{m-1}M^n_{m-1}
+\sum_{i=m}^{n-1} (g_{i+1}-g_i)x^nk_iM^n_i\,.
\end{multline}
\end{lemma}


\begin{proof}
Equations \eqref{siltrunc} have the property that allows us to apply the argument in the proof of Theorem III-4-5 in \cite{h&s} 
to prove the nonnegativity property, that is, by adding any $\varepsilon>0$ to all equations, and being 
$y^{n,\varepsilon}$ the corresponding solution, for any $t_0\geqslant 0,$ such that $y^{n,\varepsilon}(t_0)\in (\mathbb{R}_0^+)^{n+2}$ 
and, for some $j\in\mathbb{N},$ $y^{n,\varepsilon}_j(t_0)=0$, then, $\dot{y}^{n,\varepsilon}_j(t_0)>0$.

\medskip

Since \eqref{ynorm1} has been proved before, then we can say that $y^n$ is 
bounded in any bounded subset of its maximal domain of existence. This, together
 with the smoothness of the vector field in $\mathbb{R}^{n+2}$, implies that this solution is extendable to $[0,+\infty).$

\medskip

Equation \eqref{moments} is easily obtained by multiplying the equation for each $\dot{M}^n_i$ in \eqref{siltrunc} by $g_i$ and subsequently adding them together.
\end{proof}



\section{Existence}

We consider the existence theorem for the Cauchy problem \eqref{silsys} for a relevant class of rate coefficients and initial conditions.
In the proof we start by obtaining a candidate to a local solution as the limit of a subsequence of solutions of the truncated systems. 
This is done by establishing estimates that allow us to apply Ascoli-Arzelà theorem. Then, we pass to the limit in the integral version 
of the truncated system by using some a priori estimates obtained via Gronwall's inequality.
These ideas were originally used in the seminal paper \cite{BCP} on the Becker-D\"oring system.   

\begin {theorem}\label{existence}
Let $(g_i)_{i\in\mathbb{N}_0}$ and $(k_i)_{i\in\mathbb{N}_0}$ be nonnegative sequences such that, for some $\delta>0,$ 
$g_{i+1}-g_i\geqslant \delta$, $i=0,1,2,\dots$, and furthermore,
$$(g_{i+1}-g_i)k_i=O(g_i),\quad\text{for}\quad i=0,1,2,\dots. $$
Let $(p_i)_{i\in\mathbb{N}_0}$ and $(q_i)_{i\in\mathbb{N}_0}$ be two arbitrary nonnegative sequences.

Then, for $y_0=(x_0,M_{00},M_{01},\dots)\in X_+$, such that, $\sum_{i=0}^\infty g_iM_{0i}<\infty$, 
there exists a solution $y=(x,M_{0},M_{1},\dots)$ of \eqref{silsys} on $[0,+\infty)$ with $y(0)=y_0$, 
that satisfies, for any $T\in (0,\infty)$,
\begin{equation} \label{solest}
\sup_{t\in [0,T]}\sum_{i=0}^\infty g_iM_i(t)<\infty,\qquad
\int_0^T\sum_{i=1}^\infty g_i(p_i+q_i)M_i(s)\,ds<\infty\,.
\end{equation}
\end{theorem}

\medskip

\begin{proof}
	Let $y^n(0)=(x_0,M_{00},M_{01},\dots, M_{0n})$. By Lemma \ref{truncexist} there exists a unique solution $y^n$ of \eqref{siltrunc}
           defined on $[0,\infty)$, with $y^n_r(t)\geqslant 0$ for $1\leqslant r \leqslant n+2,$ and satisfying 
	$
	\mathcal{U}^n(t)\leqslant \mathcal{U}^n(0)+(r+\alpha)t,
	$
 	for $t\geqslant 0$.

\medskip

We shall regard $y^n$ as an element of $X_+$, by considering $y^n_r(t)\equiv 0$, if $r> n+2$, or equivalently $ M^n_r(t) \equiv 0$ if $r>n$.
With this identification we have $\|y^n\|=\mathcal{U}^n$, and therefore,  for all $t\geqslant 0,$ 
	\begin{equation}\label{ynnorm}
	\|y^n(t)\|\leqslant \|y^n(0)\|+(\alpha+r)t\leqslant \|y_0\|+(\alpha+r)t.
	\end{equation}

	\medskip
	
	Fix $T\in \mathbb{R}^+$ arbitrarily. 
	
	\medskip
	
	Since, for each $i=0,1,2,\dots,n$, we obviously have $(i+1)M^n_i(t)\leqslant \|y^n(t)\|,$ we conclude that, for all $t\in[0,T]$,
	\begin{equation}\label{Minest}
	0\leqslant M^n_i(t)\leqslant \frac{\|y_0\|}{1+i}+\frac{\alpha+r}{1+i}t\leqslant  \frac{\|y_0\|}{1+i}+\frac{\alpha+r}{1+i}T\,.
	\end{equation}
	
	Now, in order to apply Ascoli-Arzel\`a theorem, we need estimates on $|\dot{M}^n_i|$, $i=0,1,\dots,n$. From \eqref{siltrunc}
we have, for all $t> 0$,
	\begin{equation}\label{dotMinest}
	\begin{aligned}
	|\dot{M}^n_0(t)|&\leqslant r + (k_0 x^n(t)+(p_0+q_0))M_0^n(t),\\
	|\dot{M}^n_i(t)|&\leqslant k_{i-1} x^n(t)M^n_{i-1}(t)+k_i x^n(t) M^n_i(t)+(p_i+q_i)M_{i}^n(t),\\
	|\dot{M}^n_n(t)|&\leqslant k_{n-1} x^n(t)M^n_{n-1}(t)+(p_n+q_n)M_{n}^n(t)\,.
	\end{aligned}
	\end{equation} 
By the estimates \eqref{ynnorm}, \eqref{Minest} and \eqref{dotMinest} with $t\in [0,T]$, and taking into account 
that $x^n\leqslant \|y^n\|$, we conclude
\begin{equation}\label{dotMinestT}
	\begin{aligned}
	|\dot{M}^n_0(t)|&\leqslant  r+\left[k_0(\|y_0\|+(\alpha+r)T)+p_0+q_0\right](\|y_0\|+(\alpha+r)T),\\
	|\dot{M}^n_i(t)|&\leqslant
	 \left[(\dfrac{k_{i-1}}{i}+\dfrac{k_i}{i+1})(\|y_0\|+(\alpha+r)T)+\dfrac{p_i+q_i}{i+1}\right]\times\\
	&\hspace*{5mm}\times (\|y_0\|+(\alpha+r)T), \quad\text{for } i=1,2,\dots,n.
	\end{aligned}
	\end{equation}
	An important fact to be retained from estimates \eqref{dotMinestT} is that, for each $i=0,1,2,\dots$, $|\dot{M}^n_i|$ 
is bounded uniformly in $n$, and therefore the set of functions $M^n_i(\cdot)$, $n=1,2,\dots,$  forms an equibounded and 
equicontinuous set of functions in $[0,T].$  Therefore, by applying Ascoli-Arzel\`a theorem, and by a standard diagonalization 
procedure, we can establish the existence, for each $i=0,1,2,\dots$, of a continuous function $M_i:[0,T]\to \mathbb{R}$, 
and a strictly increasing sequence of natural numbers $(n_j)$, such that, 
$$
M_ i^{n_j}\to M_i,\text{ uniformly on  } [0,T], \text{ as }j\to\infty\,.
$$ 
Obviously, $M_i(t)\geqslant 0$, for all $t\in [0,T].$ On the other hand, since by \eqref{ynorm1}
$$
\sum_{i=0}^N (i+1)M_i(t)=\lim_{j\to\infty}\sum_{i=0}^N (i+1)M^{n_j}_i(t)\leqslant \|y_0\|+(\alpha+r)T,
$$
for each finite $N$, we conclude that, for $t\in[0,T],$
\begin{equation}\label{iMiest}
\sum_{i=0}^\infty (i+1)M_i(t)\leqslant \|y_0\|+(\alpha+r)T\,.
\end{equation}

\medskip

Considering now the sequence $(x^n(\cdot))$, observe that the equation for $\dot{x}^n$ in system \eqref{siltrunc} contains 
sums which are not clear how to control  by estimates similar to the  ones above as $n\to \infty$. 
Therefore, we adopt the following different approach already used in the proof of Theorem 2.2 in \cite{BCP}:  Since, for $t\in[0,T]$, 
\begin{equation}\label{xnest}
|x^n(t)|\leqslant \|y^n(t)\|\leqslant \|y_0\|+(\alpha+r)T,
\end{equation}
we can extract a subsequence of $(n_j)$, which we still designate by $(n_j),$ such that, for some non-negative $x\in L^{\infty}(0,T)$,
$$
x^{n_j}	\stackrel{\ast}{\rightharpoonup} x\;\; \text{ in }L^\infty(0,T), \quad\text{ as } j\to\infty\,,
$$
which means that
\begin{equation}\label{weakstar}
\forall \phi\in L^1(0,T),\quad\int_0^T(x^{n_j}(t)-x(t))\phi(t)\,dt \to 0, \quad \text{ as } j\to\infty\,.
\end{equation}

\medskip

Now, from \eqref{siltrunc}, we have that, for all $t\geqslant 0$,
\begin{equation}\label{Mnk0}
M^{n_j}_0(t)=M_{00}+rt-\int_0^t\left[ k_0x^{n_j}(s)M^{n_j}_0(s)-(p_0+q_0)M^{n_j}_0(s))\right]ds\,.
\end{equation}
Since, for all $t\in [0,T],$
$$
\int_0^t x^{n_j}(s)M^{n_j}_0(s)\,ds =
\int_0^t x^{n_j}(s)M_0(s)\,ds+\int_0^t x^{n_j}(s)(M^{n_j}_0(s)-M_0(s))\,ds,
$$
using the convergence of $M^{n_j}_0\to M_0$, uniformly on $[0,T]$, and the weak$^*$ convergence \eqref{weakstar} 
combined with estimate \eqref{xnest}, we can pass to the limit in  \eqref{Mnk0} to obtain, for all $t\in [0,T],$
$$
M_0(t)=M_{00}+rt-\int_0^t\left[ k_0x(s)M_0(s)-(p_0+q_0)M_0(s))\right]ds\,,
$$
thus proving \eqref{M0t} in Definition~\ref{defsol}. Now, fix $i\in\mathbb{N}$. For $j$ sufficiently large, by \eqref{siltrunc},
\begin{multline}\label{Mnki}
M^{n_j}_i(t)=\\=M_{0i}+\int_0^t\left[k_{i-1}x^{n_j}(s)M^{n_j}_{i-1}(s)
-k_{i}x^{n_j}(s)M^{n_j}_{i}(s)-(p_i+q_i)M^{n_j}_i(s)\right]ds\,.
\end{multline}
Using the same arguments, now with the uniform convergence $M^{n_j}_i\to M_i$, on $[0,T]$, $i=0,1,2,\dots$, 
we can pass to the limit in \eqref{Mnki} to obtain
\begin{equation*}
M_i(t)=M_{0i}+\int_0^t\left[ k_{i-1}x(s)M_{i-1}(s)-k_{i}x(s)M_{i}(s)-(p_i+q_i)M_i(s)\right]ds\,,
\end{equation*}
thus proving \eqref{Mit} in Definition \ref{defsol}. 

\medskip

In order to prove the other assertions in Definition \ref{defsol} and, in particular, 
to pass to the limit in the integral version of the last equation of \eqref{siltrunc} 
we need some further estimates. Using estimates \eqref{iMiest} and \eqref{xnest} in 
the moments' equation \eqref{moments} with $m=1$, and using the hypothesis of the theorem, 
we obtain\footnote{We recall the reader that for the sequences $(M_i^{n_j})$ we are considering in this
proof we have $M_i^{n_j}=0$ if $i>n_j$ and hence all the sums up to $\infty$ that involve these sequences
have, in fact, only a finite number of terms with $i$ up to $i=n_j$.}, after integration, for $t\in [0,T],$
\begin{multline}
\sum_{i=1}^{\infty} g_iM^{n_j}_i(t)+\int_{0}^t\sum_{i=1}^{\infty} g_i(p_i+q_i)M^{n_j}_i(s)\,ds\leqslant 
C_1+C_2\int_{0}^t\sum_{i=1}^{\infty} g_iM^{n_j}_i(s)\,ds\,, \label{novaeq}
\end{multline}
where $C_1:= k_0g_1(C_2/C)^2T+\sum_{i=1}^{\infty} g_iy_{0i},$ $C_2:= \|y_0\|+(\alpha + r)T,$ and
$C$ is such that $|g_{i+1}-g_i|k_i \leqslant Cg_i.$ 
Then, due to the nonnegativity of $(g_{i})$, $(p_i)$, $(q_i)$ and $(M_i^{n_j})$, we
can apply Gronwall's lemma to the above inequality to deduce that, for all $t\in [0,T],$
\begin{equation}\label{gronwall}
\sum_{i=1}^{\infty} g_iM^{n_j}_i(t)+\int_{0}^t\sum_{i=1}^{\infty} g_i(p_i+q_i)M^{n_j}_i(s)\,ds\leqslant C_1e^{C_2t}.
\end{equation}
Now, consider any integer $N\geqslant 2.$ Again, by the nonnegativity referred to above, 
the bound \eqref{gronwall} implies that, for all $j$,
$$
\sum_{i=1}^{N} g_iM^{n_j}_i(t)+\int_{0}^t\sum_{i=1}^{N} g_i(p_i+q_i)M^{n_j}_i(s)\,ds\leqslant C_1e^{C_2t}\,,
$$
so that, by making $j\to\infty$, we obtain, due to the uniform convergence $M_i^{n_j}\to M_i$ in $[0,T],$ for $i=1,2,\dots,$
$$
\sum_{i=1}^{N} g_iM_i(t)+\int_{0}^t\sum_{i=1}^{N} g_i(p_i+q_i)M_i(s)\,ds\leqslant C_1e^{C_2t}\,.
$$
Then, taking $N\to\infty$, we conclude that
\begin{equation}\label{giMiest}
\sum_{i=1}^{\infty} g_iM_i(t)+\int_{0}^t\sum_{i=1}^{\infty} g_i(p_i+q_i)M_i(s)\,ds\leqslant C_1e^{C_2t}\,,
\end{equation}
where we have used the monotone convergence theorem for the limit of the integral. This clearly proves assertions \eqref{solest}. 
Using the assumptions on $(g_i)$, namely that for all $i=0,1,2,\dots$, we have $\delta k_i\leqslant (g_{i+1}-g_i)k_i\leqslant Cg_i,$
 and  by \eqref{giMiest} we can write 
$$
\int_0^t\sum_{i=0}^\infty k_iM_i(s)\,ds<\infty, \quad\text{ for all } t\in[0,T).
$$
Also, from $g_{i+1}-g_i\geq \delta>0$ we get $g_i\geqslant g_0 +i\delta$
for $i=1,2,\dots,$ and hence, by \eqref{giMiest},
$$
\int_0^t\sum_{i=0}^\infty i(p_i+q_i)M_i(s)\,ds<\infty,\quad\text{ for all } t\in[0,T).
$$
This proves the second assertion in (iii) in Definition \ref{defsol}. The first assertion of (iii) is a direct 
consequence of \eqref{weakstar} and the uniform boundedness of $x^{n_j}$, $j\in\mathbb{N}$. 
Since properties (i), (ii)  are easily obtained for the coordinates $y_i$, $i=2,3,\dots$, that is, 
for $M_i$, $i=0,1,2,\dots$, by the uniform convergence properties, what is still lacking are the 
corresponding ones for $y_1$, that is, for $x$, and \eqref{xt} in Definition \ref{defsol}.

\medskip

Our next main goal is to prove that $(x^n_j)$ is a Cauchy sequence in the uniform convergence norm. From the
integrated version of \eqref{siltrunc} and letting $\chi_{i,p}$ be the indicator function of $\{0, 1,\ldots, n_p-1\}$, we can
write, for $j>l\in\mathbb{N}$ and $t\in[0,T],$
\begin{equation}\label{xnkxnl}
\begin{aligned}
x^{n_j}(t)-x^{n_l}(t)=&
-\int_0^t \biggl(x^{n_j}(s)\sum_{i=0}^{\infty}k_iM^{n_j}_i(s)\chi_{i,j}-x^{n_l}(s)\sum_{i=0}^{\infty}k_iM^{n_l}_i(s)\chi_{i,l}\biggr)ds\\
&+\int_0^t\biggl(\sum_{i=0}^{\infty}iq_iM^{n_j}_i(s)-\sum_{i=0}^{\infty}iq_iM^{n_l}_i(s)\biggr)\,ds\\
=& \;I_1+I_2+I_3\,,
\end{aligned}
\end{equation}
where,
\begin{equation}\label{I1I2I3}
\begin{aligned}
I_1:&=-\int_0^t (x^{n_j}(s)-x^{n_l}(s))\sum_{i=0}^{\infty}k_iM^{n_j}_i(s)\chi_{i,j}\,ds\,,\\
I_2:&=-\int_0^t x^{n_l}(s)\sum_{i=0}^{\infty}k_i\left(M^{n_j}_i(s)-M^{n_l}_i(s)\right)\chi_{i,j}\,ds\,,\\
I_3:&=\int_0^t\sum_{i=0}^{\infty}iq_i\left(M^{n_j}_i(s)-M^{n_l}_i(s)\right)\,ds\,.
\end{aligned}
\end{equation}
We will deal first with $I_2$ and $I_3$.
Given $N\in \mathbb{N}$, and for $j>l$ sufficiently large, we use the splitting
$\sum_{i\geqslant 0}=\sum_{i=0}^{N-1}+\sum_{i\geqslant N}$. We start by proving that the parts of $I_2$ and $I_3$ that
correspond to the sums with $i\geqslant N$ can be made arbitrarily small by choosing $j$ and $l$ sufficiently large.
In order to achieve this we will now deduce a number of adequate estimates.

\medskip

Using again \eqref{moments} with $2 \leqslant m<n_j$ and with $g_i=1$, for all $i\in\mathbb{N}$, we can write after integration
(remembering that $M^{n_j}_i=0$ for $i>n_j$)
\begin{multline}\label{sumMi}
\sum_{i=m}^{\infty} M^{n_j}_i(t)-\sum_{i=m}^{\infty} M^{n_j}_{0i} =\\
=\int_0^t x^{n_j}(s)k_{m-1}M^{n_j}_{m-1}(s)\,ds
-\int_0^t\sum_{i=m}^{\infty} (p_i+q_i)M^{n_j}_i(s)\,ds.
\end{multline}
Next, we pass to the limit $n_j\to\infty$ in \eqref{sumMi}. Write the integrand of the first integral in the form 
$$x^{n_j}(s)M^{n_j}_{m-1}(s)=x^{n_j}(s)M_{m-1}(s)+x^{n_j}(s)\left(M^{n_j}_{m-1}(s)-M_{m-1}(s)\right).$$
By \eqref{weakstar},
$$
\int_0^t x^{n_j}(s)M_{m-1}(s)\,ds\; \to \; \int_0^t x(s)M_{m-1}(s)\,ds,\quad\text{as }j\to\infty\,.
$$
On the other hand, by estimate \eqref{xnest} and the convergence $M^{n_j}_{m-1}-M_{m-1}\to 0,$ as $n_j\to\infty,$ uniformly on $[0,T]$,
$$
\int_0^t x^{n_j}(s)\left(M^{n_j}_{m-1}(s)-M_{m-1}(s)\right)\,ds\; \to \;  0,\quad\text{as }j\to\infty\,.
$$
We can then conclude that
\begin{equation}\label{int1}
\int_0^t x^{n_j}(s)k_{m-1}M^{n_j}_{m-1}(s)\,ds\; \to \; \int_0^t x(s)k_{m-1}M_{m-1}(s)\,ds,\;\text{ as }j\to\infty\,.
\end{equation}
For the second integral in the right-hand side of \eqref{sumMi} we observe that, for every $\widetilde{m}>0$  and due to
\eqref{gronwall}, the following holds uniformly in $n_j$ and $t\in [0, T]$
\begin{equation}\label{boundM}
\begin{aligned}
\int_0^t\sum_{i=\widetilde{m}}^{\infty} (p_i+q_i)M^{n_j}_i(s)\,ds 
&\leqslant\frac{1}{\widetilde{m}}\int_0^t \sum_{i=\widetilde{m}}^{\infty} i(p_i+q_i)M^{n_j}_i(s)\,ds\\
&\leqslant \frac{1}{\widetilde{m}}\int_0^t \sum_{i=1}^{\infty} i(p_i+q_i)M^{n_j}_i(s)\,ds\\
&\leqslant \frac{C_1}{\widetilde{m}}\int_0^te^{C_2s}ds\\
& \leqslant  \frac{C_1}{\widetilde{m}C_2}\bigl(e^{C_2T}-1\bigr) \;\to \; 0\;\text{as $\widetilde{m}\to\infty$}.
\end{aligned}
\end{equation}
For every $\ve>0$ fix $\widetilde{m}$ such that the last right-hand side in \eqref{boundM} is less than $\frac{1}{2}\ve.$ By the uniform in $t$
convergence of $M^{n_j}_{i}$ to $M_{i}$ we conclude that there exists a $\nu=\nu(\widetilde{m})$ such that, for all $t\in [0, T]$ and $n_j>\nu$,
\begin{equation}\label{boundM2}
\left|\int_0^t\sum_{i=m}^{\widetilde{m}-1}(p_i+q_i)M_i^{n_j}(s)ds - \int_0^t\sum_{i=m}^{\widetilde{m}-1}(p_i+q_i)M_i(s)ds\right| < \frac{1}{2}\ve.
\end{equation}
Thus, \eqref{boundM} and \eqref{boundM2} imply that for every $\ve>0$ there exists $\widetilde{m}$ and $\nu$ such that, for all $t\in [0, T]$
and $n_j>\nu$, we have
\[
\left|\int_0^t\sum_{i=m}^{\infty}(p_i+q_i)M_i^{n_j}(s)ds - \int_0^t\sum_{i=m}^{\widetilde{m}-1}(p_i+q_i)M_i(s)ds\right| < \ve
\]
and, due to the arbitrariness of $\widetilde{m}$, we conclude that, as $n_j\to\infty$,
\begin{equation}\label{ipiqiMi}
\int_0^t\sum_{i=m}^{\infty} (p_i+q_i)M^{n_j}_i(s)\,ds\to
\int_0^t\sum_{i=m}^{\infty} (p_i+q_i)M_i(s)\,ds\,.
\end{equation}
Next we will deal with the convergence of the left-hand side of \eqref{sumMi}. Again let us consider a fixed $\widetilde{m}>m$.
Then,  for  large $j$ so that $n_j>\widetilde{m}$, write
\begin{equation}\label{sumMnkM}
\sum_{i=m}^{\infty}M^{n_j}_i(t)-\sum_{i=m}^{\infty}M_i(t)=\left(\sum_{i=m}^{\widetilde{m}-1}+\sum_{i=\widetilde{m}}^{\infty}\right)\left(M^{n_j}_i(t)-M_i(t)\right)\,.
\end{equation}
Estimates \eqref{ynnorm} and \eqref{iMiest} imply that
$$
\sum_{i=\widetilde{m}}^{\infty}\left|M^{n_j}_i(t)-M_i(t)\right|
\leqslant \frac{1}{\widetilde{m}+1}\sum_{i=\widetilde{m}}^\infty (i+1)\left(M^{n_j}_i(t)+M_i(t)\right)\leqslant \frac{2(\|y_0\|+(\alpha+r)T)}{\widetilde{m}+1}\,,
$$
so that, for each $t\in[0,T]$, the series $\sum_{i=\widetilde{m}}^\infty$ in the right-hand side of 
\eqref{sumMnkM} converges, uniformly with respect to $j\in\mathbb{N},$ which allows us to conclude that, 
$$
\lim_{j\to\infty}\left(\sum_{i=m}^{\widetilde{m}-1}+\sum_{i=\widetilde{m}}^{\infty}\right)\left(M^{n_j}_i(t)-M_i(t)\right)=0\,,
$$
and so,
\begin{equation}\label{limMnk}
\lim_{j\to\infty}\sum_{i=m}^{\infty} M^{n_j}_i(t)=\sum_{i=m}^{\infty}M_i(t)\,.
\end{equation}
Plugging \eqref{int1}, \eqref{ipiqiMi} and \eqref{limMnk} into \eqref{sumMi} we obtain
\begin{equation}\label{sumMi2}
\sum_{i=m}^{\infty} M_i(t)-\sum_{i=m}^{\infty} M_{0i}=\int_0^t \Bigl(x(s)k_{m-1}M_{m-1}(s)
-\sum_{i=m}^{\infty} (p_i+q_i)M_i(s)\Bigr)\,ds\,.
\end{equation}
We shall now use \eqref{sumMi2} to obtain an estimate that will be crucial to handle the limit of $I_2$ and $I_3$ in \eqref{I1I2I3} as $n_j\to \infty$.
For each $m\in\mathbb{N}$ and $t\geqslant 0$ define
\begin{equation} \label{fm}
f_m(t):=g_m\times \Bigl(\sum_{i=m}^{\infty} M_i(t)-\sum_{i=m}^{\infty} M_{0i}\Bigr)\,,   
\end{equation}
where $(g_m)$ is a positive sequence satisfying the assumptions of the theorem.
We can write the product of \eqref{sumMi2} by $g_m$ in the following form:
\begin{equation}\label{gmfm}
f_m(t) = g_m\int_0^t\Bigl( k_{m-1}x(s)M_{m-1}(s)
- \sum_{i=m}^{\infty}(p_i+q_i)M_i(s)\Bigr)\,ds\,.
\end{equation}
Now, using the assumption $g_{i+1}-g_i\geqslant \delta>0$, we have, for $t\in[0,T]$,
$$
|f_m(t)|\leqslant \sum_{i=m}^{\infty}g_i M_i(t)+\sum_{i=m}^{\infty} g_iM_{0i}\,,
$$
and, by \eqref{giMiest}, we conclude that, for each $t\in[0,T],$
$$
|f_m(t)|\leqslant C\;,\quad\text{and}\quad \lim_{m\to\infty} f_m(t)=0,
$$
where $C>0$ is independent of $t$ and $m$. Hence, by the bounded convergence theorem, we have
$$
\lim_{m\to\infty} \int_0^T |f_m(t)|\,dt=0\,
$$ 
and so, given $\varepsilon >0$, there exists $N>2$ such that
$$
\int_0^T |f_N(t)|<\varepsilon\;\quad\text{and}\quad\; \sum_{i=N}^{\infty} g_iM_{0i}<\varepsilon\,.
$$
Using this in \eqref{gmfm}, we obtain, for all $t\in[0,T],$
$$
g_N\int_0^t\int_0^s\biggl( x(\tau)k_{N-1}M_{N-1}(\tau)
-  \sum_{i=N}^{\infty}(p_i+q_i)M_i(\tau)\biggr)\,d\tau ds<\varepsilon\,.
$$
Now define $f_m^{n_j}(t)$ by \eqref{fm} using the sequence $(M_i^{n_j})$ instead of $(M_i)$. Thus, we
have the following equality corresponding to \eqref{gmfm},
\begin{equation}\label{gmfmj}
f_m^{n_j}(t) = g_m\int_0^t\Bigl( k_{m-1}x^{n_j}(s)M_{m-1}^{n_j}(s)
- \sum_{i=m}^{\infty}(p_i+q_i)M_i^{n_j}(s)\Bigr)\,ds\,.\nonumber
\end{equation}
Using \eqref{int1} and \eqref{ipiqiMi} we conclude that
\begin{align*}
& \lim_{n_j\to \infty} g_m\int_0^t\Bigl( k_{m-1}x^{n_j}(s)M_{m-1}^{n_j}(s)
- \sum_{i=m}^{\infty}(p_i+q_i)M_i^{n_j}(s)\Bigr)\,ds = \\ 
& = g_m\int_0^t\Bigl( k_{m-1}x(s)M_{m-1}(s)
- \sum_{i=m}^{\infty}(p_i+q_i)M_i(s)\Bigr)\,ds,
\end{align*}
and thus $f_m^{n_j}(t) \to f_m(t),$ pointwise in $[0, T].$ 
Additionally, from the definition of
$f_m^{n_j}(t)$ and \eqref{gronwall} we have
\[
\left|f_m^{n_j}(t)\right| \leqslant \sum_{i=m}^{\infty} g_iM_i^{n_j}(t) + \sum_{i=m}^{\infty} g_iM_{i0} \leqslant 2C_1e^{C_2T},
\]
and consequently, by the dominated convergence theorem,
\[
\int_0^t f_m^{n_j}(s)ds \to \int_0^t f_m(s)ds\quad\text{as $n_j\to\infty$}.
\]
Hence, for every $\varepsilon>0$ there exists $m_0, j_0$ such that, for all $m>m_0$ and $n_j>j_0$ the following hold
\begin{multline}\label{2epsilon}
g_N\int_0^t\int_0^s\biggl( x^{n_j}(\tau)k_{N-1}M^{n_j}_{N-1}(\tau)
-  \sum_{i=N}^{n_j}(p_i+q_i)M^{n_j}_i(\tau)\biggr)\,d\tau ds = \\
=\, \int_0^t  f_m^{n_j}(s)ds \,=\, \int_0^t  f_m(s)ds + \int_0^t  \left(f_m^{n_j}(s)-f_m(s)\right)ds\,<\,2\varepsilon\,.
\end{multline}
Now, we consider the integrated version of the moments' equation \eqref{moments} with $m=N.$ 
Taking into consideration the hypothesis of the theorem, namely that $k_i(g_{i+1}-g_i)=O(g_i) \leqslant K_gg_i$ for some constant $K_g$
independent of $i$, and recalling \eqref{ynnorm}, we can write, for all $t\in [0,T]$, 
\begin{align*}
\int_0^t\biggl(\sum_{i=N}^{\infty} g_iM^{n_j}_i(s)&-\sum_{i=N}^{\infty} g_iM^{n_j}_i(0)+\int_0^s\sum_{i=N}^{\infty} g_i(p_i+q_i)M^{n_j}_i(\tau)\,d\tau\biggr)ds\;= \\
&=\int_0^t \sum_{i=N}^{\infty} g_iM^{n_j}_i(0)\,ds + g_N \int_0^t\int_0^sk_{N-1}x^{n_j}(\tau)M_{N-1}^{n_j}(\tau)d\tau ds\; + \\
&\hspace*{2.5cm} + \int_0^t\int_0^s\sum_{i=N}^\infty k_{i}(g_{i+1}-g_i)x^{n_j}(\tau)M_{i}^{n_j}(\tau)d\tau ds  \\
& \leqslant t \sum_{i=N}^{\infty} g_iM^{n_j}_i(0) + g_N \int_0^t\int_0^sk_{N-1}x^{n_j}(\tau)M_{N-1}^{n_j}(\tau)d\tau ds\; + \\
&\hspace*{2.5cm} + K_gC_2 \int_0^t\int_0^s\sum_{i=N}^\infty g_{i}x^{n_j}(\tau)M_{i}^{n_j}(\tau)d\tau ds
\end{align*}
where, we recall the reader, $C_2:=\|y_0\|+(\alpha+r)T$. Let  
$$
z(t):=\int_0^t\biggl(\sum_{i=N}^{\infty}g_iM^{n_j}_i(s)+\int_0^s \sum_{i=N}^{\infty} g_i(p_i+q_i)M^{n_j}_i(\tau)\,d\tau\biggr)ds\,.
$$
The nonnegativity of $M_i^{n_j}$, the inequality just obtained, and \eqref{2epsilon} imply that, for all $t\in [0,T],$
$$
z(t)\leqslant t\sum_{i=N}^{\infty} g_iM_{i}^{n_j}(0)+2\varepsilon+K_gC_2\int_0^t z(s)\,ds\,.
$$
Applying Gronwall's inequality we get, with $K_1=K_gC_2$ and for all $N$ sufficiently large,
\begin{align*}
z(t)&\leqslant t\sum_{i=N}^{\infty} g_iM_{i}^{n_j}(0)+2\varepsilon +
K_1\int_0^t \left(s\sum_{i=N}^{\infty} g_iM_{i}^{n_j}(0)+2\varepsilon\right)e^{K_1(t-s)}ds \\
&=K_1^{-1}\left(e^{K_1 t}-1\right)\sum_{i=N}^{\infty} g_iM_{i}^{n_j}(0)+2\varepsilon e^{K_1 t}\\
&<\varepsilon(2+K_{1}^{-1})e^{K_1 t}\,.
\end{align*}
Since,
\begin{align*}
\int_0^t \int_0^s \sum_{i=N}^{\infty} g_i(p_i+q_i)M^{n_j}_i(\tau)\,d\tau ds&=
\int_0^t (t-s) \sum_{i=N}^{\infty} g_i(p_i+q_i)M^{n_j}_i(s)\,ds\\
&\geqslant \frac{t}{2}\int_0^{t/2}\sum_{i=N}^{\infty} g_i(p_i+q_i)M^{n_j}_i(s)\,ds
\end{align*}
we finally conclude that, for sufficiently large fixed $N$ and for all $j\geqslant j_0$,
\begin{equation}
\label{intest}
\int_0^t\sum_{i=N}^{\infty}g_iM^{n_j}_i(s)ds
+\frac{t}{2}\int_0^{t/2} \sum_{i=N}^{\infty} g_i(p_i+q_i)M^{n_j}_i(s)\,ds\leqslant
\varepsilon(2+K_1^{-1})e^{K_1 t},
\end{equation}
for all $t\in [0,T]$.

\medskip

We have now enough estimates to deal with the expressions in \eqref{I1I2I3}. We consider the splitting $\sum_{i\geqslant 0}=\sum_{i=0}^{N-1}+\sum_{i\geqslant N}$. 
Using the bounds on $M^{n_j}_i$, the uniform convergence $M^{n_j}_i\to M_i$ on $[0,T]$, for each $i=0,1,2,\dots$, and the weak-star convergence  $x^{n_j}	\stackrel{\ast}{\rightharpoonup} x,$ it is easy to check that the $\sum_{i=0}^{N-1}$ parts of $I_2,\,I_3$ converge to zero uniformly for $t\in[0,T]$, as $j,l\to \infty$. 
On the other hand, by the hypothesis of the theorem we have $k_i=O(g_i),$ and 
from \eqref{intest}, we conclude that, for $j\geqslant l\geqslant j_0$ and all $t\in[0,T]$
\begin{multline*}
\left|\int_0^t x^{n_l}(s)\sum_{i=N}^{\infty}k_i\left(M^{n_j}(s)-M^{n_l}(s)\right)\,ds\right| \leqslant \\
\leqslant 2C_2\int_0^t\sum_{i=N}^{\infty}k_i\left(M^{n_j}(s)+M^{n_l}(s)\right)\,ds
\leqslant K_2 \varepsilon,
\end{multline*}
for some $K_2>0$, independent of $l, j, t$. Furthermore, since $i=O(g_i)$, 
by \eqref{intest} we have, for some $K_3>0$ independent of $l$ and $j$, 
\begin{multline*}
\left|\int_0^t\sum_{i=0}^{\infty}iq_i\left(M^{n_j}_i(s)-M^{n_l}_i(s)\right)\,ds\right|
\leqslant 
\int_0^t\sum_{i=0}^{\infty}iq_i\left(M^{n_j}_i(s)+M^{n_l}_i(s)\right)\,ds\leqslant K_3 \varepsilon,
\end{multline*}
uniformly for $t\in[0,T/2]$.
 Combining the above statements, we conclude that there is $j_1\geqslant j_0$, such that, for $j\geqslant l\geqslant j_1$, and all $t\in[0,T/2]$,
$$
|I_2(t)|+|I_3(t)|\leqslant (K_2+K_3)\varepsilon\,.
$$
But, by the same reasoning, there is $K_4>0$ independent of $l, j, t$ such that, for $j\geqslant l \geqslant j_1$, and all $t\in[0,T]$,
$$
|I_1(t)|=\left|\int_0^t \left(x^{n_j}(s)-x^{n_l}(s)\right)\sum_{i=0}^{\infty}k_iM_i^{n_j}(s)\chi_{i,j}\,ds\right|\leqslant K_4\varepsilon\int_0^t |x^{n_j}(s)-x^{n_l}(s)|ds,
$$
so, by \eqref{xnkxnl}, if $K_5=K_2+K_3+K_4$, for $j\geqslant l\geqslant j_1$,
$$
\left|x^{n_j}(t)-x^{n_l}(t)\right|\leqslant K_5\left(\varepsilon+\int_0^t |x^{n_j}(s)-x^{n_l}(s)|ds\right)\,,
$$
for all $t\in[0,T/2],$ so that, by Gronwall's inequality, there is $C>0$, such that
$$
\sup_{t\in[0,T/2]}|x^{n_j}(t)-x^{n_l}(t)|\leqslant C\varepsilon.
$$
Due to the arbitrariness of $\varepsilon$ we conclude that $(x^{n_j})$ is a Cauchy sequence in $C([0,T/2])$, 
and therefore $x^{n_j}\to x$ uniformly for $t\in[0,T/2]$. This also proves the continuity of $x$ in $[0,T/2]$.

\medskip

We are now able to pass to the limit in the integrated version of the last equation of \eqref{siltrunc}, that is, in
\begin{equation}\label{xeq}
x^{n_j}(t)=x_0+\alpha t +\int_0^t \left(-x^{n_j}(s)\sum_{i=0}^{\infty} k_iM^{n_j}_i(s)\chi_{i,j}+\sum_{i=0}^{\infty} iq_iM^{n_j}_i(s)\right)ds\,,
\end{equation}
where $\chi_{i,j}$ is the indicator function of the set $\{0, 1, \ldots, n_j-1\}$.
Splitting again $\sum_{i\geqslant 0}=\sum_{i=0}^{N-1}+\sum_{i\geqslant N}$, using \eqref{intest}, and the bound on $x^{n_j},$ 
we observe that there is a constant $K_5>0$ such that, for $j\geqslant j_1$,
\begin{multline}
\left|\int_0^t \left(-x^{n_j}(s)\sum_{i=N}^{\infty} k_iM^{n_j}_i(s)\chi_{i,j}+\sum_{i=N}^{\infty} iq_iM^{n_j}_i(s)\right)ds\right| \leqslant\\
\leqslant \int_0^t \left(x^{n_j}(s)\sum_{i=N}^{\infty} k_iM^{n_j}_i(s)\chi_{i,j}+\sum_{i=N}^{\infty} iq_iM^{n_j}_i(s)\right)ds
\leqslant K_5\varepsilon\,, \label{intest3}
\end{multline}
for all $t\in[0,T/2]$. By a procedure already used above,  writing in this estimate, 
$\sum_{i\geqslant N}=\sum_{i=N}^{l-1}+\sum_{i\geqslant l}$ and making first $j\to\infty$ and then $l\to\infty$, 
using the uniform convergence of $(M^{n_j}_i)$ and $(x^{n_j})$, we conclude that
\begin{equation}
\label{intest2}
\left|\int_0^t \left(-x(s)\sum_{i=N}^{\infty} k_iM_i(s)+\sum_{i=N}^{\infty} iq_iM_i(s)\right)ds\right|\leqslant K_5\varepsilon\,,
\end{equation}
for all $t\in[0,T/2]$.  Thus, from \eqref{xeq} and \eqref{intest3}, for $j\geqslant j_1$,
$$
\left|x^{n_j}(t)-x_0-\alpha t+\int_0^t \left(x^{n_j}(s)\sum_{i=0}^{N-1} k_iM^{n_j}_i(s)-\sum_{i=0}^{N-1} iq_iM^{n_j}_i(s)\right)ds\right|\leqslant K_5\varepsilon\,,
$$
for all $t\in[0,T/2].$ Now letting $j\to \infty$ and using the same uniform convergence property we obtain, from the last inequality,
$$
\left|x(t)-x_0-\alpha t+\int_0^t \left(x(s)\sum_{i=0}^{N-1} k_iM_i(s)-\sum_{i=0}^{N-1} iq_iM_i(s)\right)ds\right|\leqslant K_5\varepsilon\,,
$$
and hence, by \eqref{intest2},
$$
\left|x(t)-x_0-\alpha t+\int_0^t \left(x(s)\sum_{i=0}^{\infty} k_iM_i(s)-\sum_{i=0}^{\infty} iq_iM_i(s)\right)ds\right|\leqslant 2K_5\varepsilon\,,
$$
for all $t\in[0,T/2]$. Since $\varepsilon$ was chosen arbitrarily small we conclude that \eqref{xt} in Definition \ref{defsol} holds for all $t\in[0,T/2]$. 

\medskip



We have showed that for any given \( T>0 \) and for each initial condition \( y_0\in X_+\) there exists a solution \( y\) defined in \( [0,T/2] \). 
Since the equations in system \eqref{M0t}--\eqref{xt} are autonomous this allows us to extend \( y \) to \([0,\infty)\).
\end{proof}
%
Setting $g_i=i$, for $i=1,2,\dots$, in the above theorem we obtain the following result:

\begin{coro}
	Let $k_i=O(i)$. Then, for each $y_0=(x_0,M_{01},M_{02},\dots)\in X_+$, there exists a solution $y$ of \eqref{silsys} on $[0,+\infty)$ with $y(0)=y_0$.
\end{coro}
 \section{The moments' equation and some a priori results}
The next theorem states an a priori equality which is an integrated version for any solution of \eqref{silsys}, 
of a result stated in Lemma \ref{truncexist}
for the solutions of the truncated version of the silicosis system. 
\begin{theorem}\label{momentsMi}
	Let $(g_i)$ be a real sequence. Let $y:=(x,M_0,M_1,\dots)$ be a solution of \eqref{silsys} on some interval $[0,T)$, $0<T\leqslant\infty.$ Take any pair $(t_1, t_2)$, such that $0\leqslant t_1<t_2<T$. Suppose that,
\begin{equation}
\label{Prop1}
\int_{t_1}^{t_2} \sum_{i=0}^\infty |g_{i+1}-g_i|k_iM_i(s)ds<\infty\,,
	\end{equation}
	and, furthermore, one of the following two sets of conditions, {\rm(A)} or {\rm(B)}, holds:
	\begin{enumerate}
		\item[(A)] $g_i=O(i)$ \quad and \quad
		$\displaystyle\int_{t_1}^{t_2} \sum_{i=0}^\infty g_i(p_i+q_i)M_i(s)ds<\infty$; 
		\item[(B)] for $p=1,2$\,,\; $\displaystyle\sum_{i=0}^\infty g_iM_i(t_p)<\infty$, and, for sufficient large $i$,
		$g_{i+1}\geqslant g_i\geqslant 0$\,.
		\end{enumerate}
		Then, for $m\geqslant 1$, the following integrated version of \eqref{moments} holds for the solution $y$:
		\begin{multline}\label{momentscompl}
		\sum_{i=m}^\infty g_iM_i(t_2)-\sum_{i=m}^\infty g_iM_i(t_1)+\int_{t_1}^{t_2}\sum_{i=m}^\infty g_i(p_i+q_i)M_i(s)\,ds\\=
		\int_{t_1}^{t_2}g_mx(s)k_{m-1}M_{m-1}(s)\,ds
		+\int_{t_1}^{t_2}\sum_{i=m}^{\infty} (g_{i+1}-g_i)x(s)k_iM_i(s)\,ds\,.
		\end{multline}
\end{theorem}
\begin{proof}
	Take any $n>m\geqslant 1$. From \eqref{Mit} we easily obtain
			\begin{multline}\label{momentsmn}
			\sum_{i=m}^n g_iM_i(t_2)-\sum_{i=m}^n g_iM_i(t_1)+\int_{t_1}^{t_2}\sum_{i=m}^n g_i(p_i+q_i)M_i(s)\,ds =\\
			=
			\int_{t_1}^{t_2}g_mx(s)k_{m-1}M_{m-1}(s)\,ds
			-\int_{t_1}^{t_2}g_{n+1}x(s)k_{n}M_{n}(s)\,ds\\
			+\int_{t_1}^{t_2}\sum_{i=m}^{n} (g_{i+1}-g_i)x(s)k_iM_i(s)\,ds\,.
			\end{multline}
			We now pass to the limit $n\to\infty$ in both sides. By the boundedness of $x$ 
and hypothesis \eqref{Prop1}, we readily conclude that, as $n\to\infty$,
			\begin{equation}
			\label{difglim}
			\int_{t_1}^{t_2}\sum_{i=m}^{n} (g_{i+1}-g_i)x(s)k_iM_i(s)\,ds\;
			\to\; \int_{t_1}^{t_2}\sum_{i=m}^{\infty} (g_{i+1}-g_i)x(s)k_iM_i(s)\,ds\,.
			\end{equation}

			Next we claim that the second integral in the right-hand side of \eqref{momentsmn} 
			vanishes in this limit: from (ii) and (iii) of Definition \ref{defsol}, 
			we readily obtain
			\begin{equation}\label{intlim1}
			\int_{t_1}^{t_2}x(s)k_nM_n(s)\,ds \to 0, \;\text{ as }\;n\to\infty\,;
			\end{equation}
			on the other hand, by setting $g_i=1$ for all $i=0,1,2,\dots$ in \eqref{momentsmn}, we obtain,
		\begin{multline*}
		\sum_{i=m}^n M_i(t_2)-\sum_{i=m}^n M_i(t_1)+\int_{t_1}^{t_2}\sum_{i=m}^n (p_i+q_i)M_i(s)\,ds =\\
		= \int_{t_1}^{t_2}x(s)k_{m-1}M_{m-1}(s)\,ds
		- \int_{t_1}^{t_2}x(s)k_{n}M_{n}(s)\,ds\,,
		\end{multline*}
		so that, from \eqref{intlim1} and again the properties (ii) and (iii) of Definition~\ref{defsol}, by letting $n\to\infty$, 
		we obtain,
		\begin{align}\nonumber
		\sum_{i=m}^\infty M_i(t_2)-\sum_{i=m}^\infty M_i(t_1)+\int_{t_1}^{t_2}\sum_{i=m}^\infty &(p_i+q_i)M_i(s)\,ds = \\
		&=
		\int_{t_1}^{t_2}x(s)k_{m-1}M_{m-1}(s)\,ds\,.	\label{difmt2t1}
		\end{align}
		Let us now consider hypothesis (A). This implies that, for $p=1,2,$
		$$
		|g_{n+1}|\sum_{i=n+1}^\infty M_i(t_p)\leqslant C(n+1)\sum_{i=n+1}^\infty M_i(t_p)\leqslant C\sum_{i=n+1}^\infty iM_i(t_p)\,,
		$$
		for some constant $C>0$, and thus
		\begin{equation}
		\label{sumgnlim}
			|g_{n+1}|\sum_{i=n+1}^\infty M_i(t_p)\to 0\,,\;\text{ as }\; n\to \infty\,.
		\end{equation} 
		Replacing $m$ by $n+1$ in \eqref{difmt2t1}, multiplying both members by $|g_{n+1}|$,  
		letting $n\to\infty$, and using (iii) of Definition~\ref{defsol} together with \eqref{sumgnlim}, we conclude that,
		\begin{equation}
		\label{secondintlim}
		\int_{t_1}^{t_2}g_{n+1}x(s)k_{n}M_{n}(s)\,ds\to 0\,,\;\text{ as }\; n\to\infty\,,
		\end{equation}
		thus proving our claim under the first condition in (A). Using the second condition in (A) 
		together with Definition~\ref{defsol}, \eqref{difglim}, \eqref{secondintlim}, and the 
		bounded convergence theorem, we can pass to the limit $n\to\infty$ in \eqref{momentsmn} and prove \eqref{momentscompl}.
		
		\medskip
		
		Now we prove the claim under the hypothesis (B). This implies that, for some constant $C>0$,
		$$
		|g_{n+1}|\sum_{i=n+1}^\infty M_i(t_p)\leqslant C\sum_{i=n+1}^\infty g_iM_i(t_p)\,,
		$$
		and thus, under these conditions, also \eqref{sumgnlim} and \eqref{secondintlim} hold true. 
		Therefore, using hypothesis (B), \eqref{difglim}, and \eqref{secondintlim} we conclude that
		all terms in \eqref{momentsmn} but the integral in the 
		left-hand side converge as $n\to\infty.$ But then, using the monotone convergence theorem we conclude 
		that in \eqref{momentscompl} we also have,
		$$
		\int_{t_1}^{t_2}\sum_{i=m}^n g_i(p_i+q_i)M_i(s)\,ds\;\to\;
			\int_{t_1}^{t_2}\sum_{i=m}^\infty g_i(p_i+q_i)M_i(s)\,ds\,,
		$$
		as $n\to\infty$, thus completing the proof.
\end{proof}

Given some $y=(x,M_0,M_1,\dots)\in X_+$, recall the definitions of $\mathcal{X}(y)$, $\mathcal{M}(y)$ and $\mathcal{U}(y)$.

\begin{coro}\label{corol5.2}
	Let $y:=(x,M_0,M_1,\dots)$ be a solution of \eqref{silsys} in $[0,T)$, $0< T\leqslant \infty$. Then, for all $t\in[0,T),$ the 
following integrated version of \eqref{ynorm1} holds for that solution:
	\begin{multline}\label{Uyt0}
	\mathcal{U}(y(t))-\mathcal{U}(y(0)) =\\
	=(r+\alpha)t
	-\int_0^t\sum_{i=0}^\infty (p_i+q_i)M_i(s)\,ds
	-\int_0^t\sum_{i=1}^\infty ip_iM_i(s)\,ds\,.
	\end{multline}
	Also, for every $m\geqslant 1$,
		\begin{multline}\label{momentscomplg1}
		\sum_{i=m}^\infty M_i(t)-\sum_{i=m}^\infty M_i(0)+\int_{0}^{t}\sum_{i=m}^\infty (p_i+q_i)M_i(s)\,ds =\\
		=
		\int_{0}^{t}x(s)k_{m-1}M_{m-1}(s)\,ds\,,
		\end{multline}
		and
			\begin{multline}\label{momentscomplgi}
			\sum_{i=m}^\infty iM_i(t)-\sum_{i=m}^\infty iM_i(0)+\int_{0}^{t}\sum_{i=m}^\infty i(p_i+q_i)M_i(s)\,ds =\\
			=
			\int_{0}^{t}mx(s)k_{m-1}M_{m-1}(s)\,ds
			+\int_{0}^{t}\sum_{i=m}^{\infty} x(s)k_iM_i(s)\,ds\,.
			\end{multline}
\end{coro}

\begin{proof}
Equations \eqref{momentscomplg1} and \eqref{momentscomplgi} are obtained by setting $g_i=1$ and $g_i=i$, resp., 
in the previous theorem. The first choice trivially satisfies all conditions of the theorem. The second, if we recall Definition~\ref{defsol}, 
satisfies all statements corresponding to the hypothesis (B).
Adding the equation for $x(t)$, \eqref{xt}, to the equation \eqref{momentscomplgi} with $m=1$, we obtain,
\begin{equation}\label{Xyt}
\mathcal{X}(y(t))-\mathcal{X}(y(0))=\alpha t-\int_{0}^{t}\sum_{i=1}^\infty ip_iM_i(s)\,ds.
\end{equation}
On the other hand, adding the equation for $M_0(t)$, \eqref{M0t}, to equation \eqref{momentscomplg1} with $m=1$, we obtain,
\begin{equation}\label{Myt}
\mathcal{M}(y(t))-\mathcal{M}(y(0))=rt-\int_{0}^{t}\sum_{i=1}^\infty i(p_i+q_i)M_i(s)\,ds.
\end{equation}
Finally, adding together \eqref{Xyt} and \eqref{Myt} we obtain \eqref{Uyt0} thus completing the proof.
\end{proof}

The previous results allow us to directly deduce the following conclusion on the solution regularity:
\begin{coro}\label{unifconv}
	Let $y$ be a solution of \eqref{silsys} on an interval $[0,T)$, $0<T\leqslant \infty.$ Then, $y:[0,T)\to X$ is 
continuous and, moreover, the series $\sum_{i=1}^\infty iM_i(t)$ is uniformly convergent on compact subintervals of $[0,T).$
\end{coro}
\begin{proof}
	By Definition~\ref{defsol}, $x$ and $M_i$, $i=0,1,2,\dots$, are continuous real functions on $[0,T).$ 
	For each $n\in\mathbb{N},$ let $\displaystyle\phi_n(t):=x(t)+\sum_{i=0}^n(i+1) M_i(t)$, for all $t\in[0,T).$ 
	Each $\phi_n$ is a continuous function on $[0,T)$. Take now any $T'\in(0,T)$.
	By \eqref{Uyt0} we know that, for all $n\in\mathbb{N}$ and $t\in[0,T']$, 
	$$
		\phi_n(t)\leqslant\mathcal{U}(y(t))\leqslant\mathcal{U}(y(0))+(r+\alpha)T'\,.
	$$
	But, for each $t\in[0,T)$, the real sequence $(\phi_n(t))$ is increasing, and thus, 
	by Dini's theorem, the series $\sum_{i=0}^\infty i M_i(t)$ is uniformly convergent on $[0,T']$ 
	and the second statement of the corollary is proved.
	
	\medskip
	
	To prove the continuity of $[0,T)\ni t\mapsto y(t)\in X$, we fix any $t_0\in[0,T)$. Let $T'\in(t_0,T)$ 
	so that the series $\sum_{i=1}^\infty iM_i(t)$, and thus also the series $\sum_{i=1}^\infty (i+1)M_i(t)$ 
	is uniformly convergent on $[0,T'].$ Let $\varepsilon>0$ be given. Then, there exists $N=N(T')>1$ such that, 
	$\displaystyle\sup_{t\in[0,T']}\sum_{i=N+1}^\infty (i+1)M_i(t)<\varepsilon$. For any $t\in[0,T'],$
	\begin{align*}
	\|y(t)-y(t_0)\|&=|x(t)-x(t_0)|+\left(\sum_{i=0}^N+\sum_{i=N+1}^\infty\right)(i+1)|M_i(t)-M_i(t_0)|\\
	&\leqslant |x(t)-x(t_0)|+\sum_{i=0}^N(i+1)|M_i(t)-M_i(t_0)|+2\varepsilon\,.
	\end{align*}
	But then, by continuity of the real valued functions $x$ and $M_i,$ $i=0,1,2,\dots$, 
	$$
	0\leqslant\liminf_{t\to t_0}\|y(t)-y(t_0)\|\leqslant\limsup_{t\to t_0}\|y(t)-y(t_0)\|\leqslant 2\varepsilon\,.
	$$
	From this and the arbitrariness of $\varepsilon$ we conclude that
	$
	\lim_{t\to t_0}\|y(t)-y(t_0)\|=0\,.
	$
\end{proof}
Only from the definition of solution and without further assumptions we know, by conditions (iii)
in Definition~\ref{defsol}, that 
$\sum_{i=1}^\infty iq_iM_i(t)<\infty$, for a.e. $t$  in $(0,T)$. This means, in particular, that 
if $y(0)\in X_+$ but the previous series is divergent in $t=0$ then, generically, for $t>0$, it 
is convergent and thus the sequences $(y_n(t))_{n\in\mathbb{N}}$ will decay faster than $y(0)$
as $n\to\infty$. However, the next two propositions show that with additional hypothesis 
we can be more specific about this behaviour.
\begin{prop}\label{abscont1}
	If $(k_i)_{i\in\mathbb{N}_0}$, $(p_i)_{i\in\mathbb{N}_0}$ and $(q_i)_{i\in\mathbb{N}_0}$ 
	are nonnegative and furthermore, $k_i=O(i)$, and   $(q_i)_{i\in\mathbb{N}_0}$ is bounded, 
	then, for each solution $y=(x,M_0,M_1,\dots)$ on $[0,T)$, the function 
	$Q(\cdot):=\sum_{i=1}^\infty iq_iM_i(\cdot),$ is absolutely continuous on each 
	compact subinterval of $[0,T)$. Furthermore, if
	$(p_i)_{i\in\mathbb{N}_0}$ is bounded, then the function 
	$P(\cdot):=\sum_{i=1}^\infty ip_iM_i(\cdot)$ is also absolutely continuous on each compact 
	subinterval of $[0,T)$.
\end{prop}
\begin{proof}
Let the first set of conditions be fulfilled. Then, the sequence $g_i=iq_i$ satisfies \eqref{Prop1} 
and hypothesis (A) of Theorem \ref{momentsMi}, so that we consider the corresponding version of 
\eqref{momentscompl} with $m=1.$ Take any $T'\in (0,T)$. Then, that expression shows that, for 
each $t\in[0,T')$, $	Q(t)-Q(0)$ can be written as a sum of integrals of Lebesgue integrable 
functions on $[0,T')$. Therefore, it is an absolutely continuous function of $t$ on $[0,T')$. 
The same applies to $P(\cdot)$ if we consider the more strict hypothesis on the $p_i$ coefficients.
\end{proof}

\medskip

The previous Proposition imposes too stringent restrictions to the sequences $(p_i)$ and $(q_i)$ 
that might not be satisfied in some applications. Having this in mind we can still obtain the 
following result without those restrictive hypothesis at the expense of not guaranteeing 
the absolute continuity of $Q(\cdot)$ down to $t=0.$ 

\begin{prop}\label{abscont2}
Let $(k_i)_{i\in\mathbb{N}_0}$ and $(q_i)_{i\in\mathbb{N}_0}$ satisfy,
$(i+1)q_{i+1}\geqslant iq_i,$ for sufficiently large $i$, and $((i+1)q_{i+1}-iq_i)k_i=O(iq_i)$. Consider a solution $y$ on
some interval $[0,T)$. Then, the function $Q(\cdot)$ is absolutely continuous on each compact subinterval of $(0,T)$.
\end{prop} 
\begin{proof}
	Take any pair $(t_1,t_2)$ such that, $0<t_1<t_2<T'$, for some fixed $T'\in(0,T)$, and $\sum_{i=1}^\infty iq_iM_i(t_p)$
 is finite for $p=1,2.$ We claim that $g_i=iq_i$ satisfies \eqref{Prop1} and hypothesis (B) of Theorem \ref{momentsMi}. Since
 $((i+1)q_{i+1}-iq_i)k_i=O(iq_i)$ and by (ii) and (iii) of Definition~\ref{defsol}, we can conclude that \eqref{Prop1} is satisfied. 
On the other hand, the hypothesis also states that, for our choice of the sequence $(g_i)_{i\in\mathbb{N}_0}$, we
 have $g_{i+1}\geqslant g_i$, for sufficiently large $i$. This, together with the way $t_1$ and $t_2$ were chosen, 
proves our claim. This establishes the validity of \eqref{momentscompl} with $m=1$
	for this choice of $t_1,t_2$ and $g_i=iq_i$, $i=0,1,2,\dots$.  But the first integral on the 
	right-hand side of \eqref{momentscompl} is bounded by a constant depending only on $T'$, and, by the hypothesis, 
	for some $C>0$ independent of the particular choice of $t_1,t_2$, we have
	$$
	\int_{t_1}^{t_2}\sum_{i=1}^{\infty} (g_{i+1}-g_i)x(s)k_iM_i(s)\,ds\leqslant
	C\int_{t_1}^{t_2}\sum_{i=1}^{\infty}g_iM_i(s)\,ds\,.
	$$
	Therefore, by \eqref{momentscompl}, there is another positive constant $C$ depending only of $T'$ such that,
	$$
	Q(t_2)-Q(t_1)\leqslant C\left(1+\int_{t_1}^{t_2}Q(s)\,ds\right)\,,
	$$
	so that, by fixing $t_0$ for which $Q(t_0)$ is finite, we have, by Gronwall inequality, 
	$$
	0\leqslant Q(t)\leqslant (C+Q(t_0))(1+CT'e^{CT'}),
	$$
	for $t$ a.e. in $(0,T')$. Thus, $Q\in L^\infty(0,T')$.
	Using this fact again in \eqref{momentscompl} together with our hypothesis, we conclude that $Q$ is, 
	in fact, in $C((0,T'))$ and it is absolutely continuous on this interval. This completes the proof.
\end{proof}
\section{Differentiability}
We state the first order differentiability properties of solutions for the silicosis system under two different sets of assumptions. 
In the first one, we can draw a conclusion slightly stronger than in the second one.

\begin{prop}\label{diff1}
Let $(k_i)_{i\in\mathbb{N}_0}$ be a nonnegative sequence such that $k_i=O(i).$ Let $(q_i)_{i\in\mathbb{N}_0}$ be bounded. 
Let $y$ be a solution of \eqref{silsys} in some interval $[0,T)$, $0<T\leqslant \infty.$ Then,
$x\equiv y_1$ and $M_i\equiv y_{i+1},$ $i\in\mathbb{N}_0$, are $C^1$ functions on $[0,T)$ and satisfy the equations 
\eqref{silsys} in the classical sense that is, in their differential form, for all $t\in[0,T)$.
\end{prop}

\begin{proof}
That $M_i$, $i\in\mathbb{N}_0$, are $C^1$ in $[0,T)$ is clear from the continuity of $x$ and $M_i$  
stated in (ii) of Definition~\ref{defsol} and the equations \eqref{M0t}, \eqref{Mit} from (iv) in the same definition. 
It remains to prove the regularity of $x$. 
The continuity of $x\sum_{i=1}^\infty k_iM_i$ is a consequence of the continuity of $x$, 
the assumptions, and Corollary~\ref{unifconv}. The continuity of $Q=\sum_{i=1}^\infty iq_iM_i$, 
under these hypotheses was stated in Proposition \ref{abscont1}. By \eqref{xt} in Definition~\ref{defsol} 
we conclude that $x\in C^1([0,T))$.
\end{proof}

\begin{prop}\label{diff2}
	Let $(k_i)_{i\in\mathbb{N}_0}$ and $(q_i)_{i\in\mathbb{N}_0}$ satisfy the 
	hypothesis of Proposition~\ref{abscont2} and $k_i=O(i)$.  
	Let $y$ be a solution of \eqref{silsys} in some interval $[0,T)$, $0<T\leqslant \infty.$ 
	Then, $x\equiv y_1$ and $M_i\equiv y_{i+1},$ $i\in\mathbb{N}_0$, are $C^1$ on $(0,T)$ and 
	they satisfy equations \eqref{silsys} in the classical sense that is, in their differential form, 
	for all $t\in(0,T)$.
\end{prop}
\begin{proof}
	This is just a consequence of the same arguments used in the proof of the previous 
	proposition together with Proposition~\ref{abscont2}.
\end{proof}
\begin{coro}
	Let $(k_i)_{i\in\mathbb{N}_0}$, $(p_i)_{i\in\mathbb{N}_0},$ and
	$(q_i)_{i\in\mathbb{N}_0}$ be nonnegative sequences such that,
	$k_i=O(i)$, and, for sufficiently large $i$, $q_i=a i^\eta$, 
	for some $a, \eta>0$. Then, the same conclusions in Proposition~\ref{diff2} hold.
\end{coro}

\begin{proof}
	From the fact that,
	as $i\to\infty,$
	$$
	\frac{(i+1)^{\eta+1}-i^{\eta+1}}{i^{\eta+1}}\, i \to \eta+1\,,
	$$
 and from Proposition \ref{diff2}, the result immediately follows. \end{proof}
\section{Uniqueness}
\begin{theorem}\label{uniqtheo}
	Let $(k_i)_{i\in\mathbb{N}_0}$ and $(g_i)_{i\in\mathbb{N}_0}$ be as in the hypothesis of the 
	existence Theorem~\ref{existence}. Assume furthermore that $k_i=O(i)$ and $ik_i=O(g_i)$. 
	Let also $(p_i),(q_i)$ be nonnegative sequences such that, $(p_i)_{i\in\mathbb{N}_0}$ is 
	bounded and $q_i=O(i)$. Then, for each  $y_0=(x_0,M_{00},M_{01},\dots)\geqslant 0$, 
	satisfying $\sum_{i=0}^\infty g_iy_{0i}<\infty$, there is exactly one solution 
	$y=(x,M_0,M_1,\dots)$ of \eqref{silsys} on $[0,\infty)$ satisfying $y(0)=y_0$.
\end{theorem}
\begin{proof}
	Let $y$ be a solution as stated above, proved to exist in Theorem~\ref{existence}, 
	and let $\tilde{y}$ be another solution with $\tilde{y}(0)=y_0.$ Define $z(t)=y(t)-\tilde{y}(t)$. 
	Fix any finite $T>0$. By Definition~\ref{defsol} all the coordinates of $y,\tilde{y}$ and, 
	therefore, of $z$ are absolutely continuous functions on $[0,T]$. Hence, $t\mapsto |z(t)|$ has 
	the same property and furthermore, for $t\in[0,T]$ a.e., 
	$$\frac{d}{dt}|z(t)|=(\sgn z(t))\frac{dz}{dt}(t).$$
	We then know that the differential versions of equations \eqref{M0t} and \eqref{Mit} from the 
	definition of solution, Definition~\ref{defsol}, are satisfied for a.e. $t\in[0,T]$.
	Fix any finite $N\geqslant 2.$ Having in mind that $y_1=x$ and  $y_i=M_{i-2}$, $i=2,3,\dots$, and 
	the same for $\tilde{y}$, rewriting each one of those equations using $y_i,\tilde{y}_i$ notation, we can obtain, for a.e. $t\in[0,T]$,
	\begin{align}
	\nonumber\frac{d}{dt}\sum_{i=0}^{N-2}(i+1)|M_i-\tilde{M}_i| 
	&=\;\frac{d}{dt}\sum_{i=2}^{N}(i-1)|z_i|\\
	\nonumber
	&=\;\sum_{i=2}^{N}(i-1)\frac{dz_i}{dt}\sgn(z_i)\\
	&=\; \sum_{i=2}^N k_{i-2}(y_1y_{i}-\tilde{y}_1\tilde{y}_{i})\left[i\sgn(z_{i+1})-(i-1)\sgn(z_i)\right]\nonumber \\ 
	&\quad -Nk_{N-2}(y_1y_N-\tilde{y}_1\tilde{y}_N)\sgn(z_{N+1})\nonumber\\ \label{uniq1}
	&\quad -\sum_{i=2}^N (i-1)(p_{i-2}+q_{i-2})|z_i|.
	\end{align}
	Now for the first sum in the right-hand side of the last equation we observe that, for $i=2,3,\dots$,
	\begin{align*}
(y_1y_{i}&-\tilde{y}_1\tilde{y}_{i})\left[i\sgn(z_{i+1})-(i-1)\sgn(z_i)\right] =\\ 
= &\;\; (\tilde{y}_1z_i+y_iz_1)\left[i\sgn(z_{i+1})-(i-1)\sgn(z_i)\right]\\
= &\;\; (\tilde{y}_1|z_i|\sgn(z_i)+y_iz_1)\left[i\sgn(z_{i+1})-(i-1)\sgn(z_i)\right]\\ 
= &\;\; \tilde{y}_1|z_i|\left[i\sgn(z_iz_{i+1})-(i-1)\right]+y_iz_1\left[i\sgn(z_{i+1})-(i-1)\sgn(z_i)\right]\\ 
\leqslant &\;\; \tilde{y}_1|z_i|+y_i|z_1|(2i-1)\,,
	\end{align*}
and using this in \eqref{uniq1} we obtain, after integrating in $[0,t]$, for every $t\in[0,T]$,
\begin{align}\nonumber
\sum_{i=2}^{N}(i-1)|z_i(t)|\leqslant& \int_0^t \biggl[\tilde{y}_1\sum_{i=2}^N k_{i-2}|z_i|+|z_1|\sum_{i=2}^N k_{i-2}(2i-1)y_i\biggr]ds\\
&-\int_0^t Nk_{N-2}(y_1y_N-\tilde{y}_1\tilde{y}_N)\sgn(z_{N+1})ds\,.\label{uniq2}
\end{align}
Since, by hypothesis, $ik_i=O(g_i)$, and $y$ satisfies \eqref{solest}, we know that
\begin{equation}\label{uniq3}
\sup_{t\in [0,T]}\sum_{i=2}^\infty k_{i-2}(2i-1)y_i<\infty\,.
\end{equation}
By \eqref{Uyt0},
$$
z_1(t)+\sum_{i=2}^\infty(i-1)z_i(t)=-\int_0^t\sum_{i=2}^\infty((i-1)p_{i-2}+q_{i-2})z_i(s)\,ds,
$$
and using the hypothesis on $(p_i)$ and $(q_i)$, we obtain, for all $t\in[0,T],$
\begin{equation}\label{z1ziintzi}
|z_1(t)|\leqslant \sum_{i=2}^\infty(i-1)|z_i(t)|
+C_0\int_{0}^t \sum_{i=2}^\infty(i-1)|z_i(s)|\,ds,
\end{equation}
for some positive constant $C_0$. Therefore, for $t\in[0,T],$
\begin{align}
\int_0^t |z_1(s)|ds&\leqslant 
\int_0^t \sum_{i=2}^\infty(i-1)|z_i(s)|\,ds+C_0
\int_0^t \int_0^s \sum_{i=2}^\infty(i-1)|z_i(\tau)|\,d\tau ds \nonumber\\
&=\int_0^t (1+C_0(t-s) )\sum_{i=2}^\infty(i-1)|z_i(s)|\,ds \nonumber\\
&\leqslant C_1\int_0^t \sum_{i=2}^\infty(i-1)|z_i(s)|\,ds\,,\label{z1zi}
\end{align}
for some positive ($T$-dependent) constant, $C_1$.

\medskip

For the last term in \eqref{uniq2}, we use \eqref{momentscomplgi} to prove that it 
converges to zero as $N\to \infty.$ In fact, from the definition of solution and the 
bounded convergence theorem we have that
\begin{multline*}
\sum_{i=m}^\infty iM_i(t)-\sum_{i=m}^\infty iM_i(0)+\int_{0}^{t}\sum_{i=m}^\infty i(p_i+q_i)M_i(s)\,ds\\
-\int_{0}^{t}\sum_{i=m}^{\infty} x(s)k_iM_i(s)\,ds\,\to \, 0\,,\quad\text{as}\quad m\to \infty\,,
\end{multline*}
and thus,
$$
\int_{0}^{t}mx(s)k_{m-1}M_{m-1}(s)\,ds\;\to\; 0,\quad\text{as}\quad m\to\infty\,,
$$
and this is equivalent to
$$
\int_0^t Nk_{N-2}y_1(s)y_N(s)ds\;\to\; 0,\quad\text{as}\quad N\to\infty\,.
$$
Proceeding similarly with respect to $\tilde{y}$, and since
$$
\left|\int_0^t Nk_{N-2}(y_1y_N-\tilde{y}_1\tilde{y}_N)\sgn(z_{N+1})ds\right|
\leqslant \int_0^t Nk_{N-2}y_1y_N ds+\int_0^t Nk_{N-2}\tilde{y}_1\tilde{y}_N ds,
$$ 
we conclude that
\begin{equation}\label{NkNm2}
\int_0^t Nk_{N-2}(y_1y_N-\tilde{y}_1\tilde{y}_N)\sgn(z_{N+1})ds\;\to\; 0,\quad\text{as}\quad N\to\infty\,.
\end{equation}
Using in \eqref{uniq2} the boundedness of $\tilde{y}$ and the hypothesis on $k_i$, together with \eqref{uniq3}, 
\eqref{z1zi}, and \eqref{NkNm2}, we deduce that there exits some constant $C>0$ such that, for each $t\in[0,T]$,
$$
\sum_{i=2}^{\infty}(i-1)|z_i(t)|\leqslant C\int_0^t \sum_{i=2}^{\infty}(i-1)|z_i(s)|\,ds\,,
$$
and thus, since by the definition of $z$, $z_i(0)=0$, we conclude that
$$
\sum_{i=2}^{\infty}(i-1)|z_i(t)|=0,
$$
for all $t\in [0,T]$. Hence, for $i=2,3,\dots,$ $z_i(t)=0$, which means that $y_i(t)=\tilde{y}_i(t)$. 
But then, by \eqref{z1ziintzi} we also have $z_1(t)=0$, which means that $y_1(t)=\tilde{y}_1(t)$. 
Since $T>0$ can be chosen arbitrarily large, the proof is complete.
\end{proof}

\begin{remark}
Observe that the uniqueness theorem in \cite{BCP} (Theorem~3.6) for the Becker-D\"oring equations, which
 corresponds to the uniqueness result, Theorem~\ref{uniqtheo}, in our present setting, 
does not require any type of restrictions on the fragmentation coefficients $b_r$, in contrast with our 
restrictions on the coefficients $p_i$ and $q_i.$ The reason for this is that, to estimate $y_1$ in 
terms of $y_i$, $i\geqslant 2$, that is, to obtain \eqref{z1ziintzi}, we used the mass balance equation 
\eqref{Uyt0} thus requiring some hypothesis on the coefficients $p_i, q_i$, while in \cite{BCP} to estimate 
$c_1$ in terms of the other  coordinates $c_r$, the authors were able to use the conservation of mass property 
which did not required restrictions on the growth rate of the $b_r$ coefficients.
\end{remark}

\begin{coro}\label{nicecorol}
Let $(p_i),$ $(q_i),$ $(k_i)$  be nonnegative sequences satisfying
$p_i=O(1)$, $q_i=O(i),$ and $k_i=O(i^\gamma)$ for some $\gamma \in [0,1].$  
Then, for each $y_0=(y_{0i})_{i\in\mathbb{N}}\in X_+$ such that
\begin{equation}
\sum_{i=2}^\infty (i-1)^{1+\gamma}y_{0i}<\infty\,,\label{higherinitialmoment}
\end{equation}
there is a unique local solution $y$ of the silicosis system \eqref{silsys} satisfying $y(0)=y_0$ and, moreover, 
it is extendable to $[0,\infty)$. In particular, if $(k_i)$ is bounded (i.e., if $\gamma =0$) then this uniqueness
result holds true for solutions with initial conditions with finite mass, that is, for nonnegative $y_0$ with
$$
y_{01}+\sum_{i=2}^\infty (i-1)y_{0i}<\infty\,.
$$
\end{coro}
The following result is also a consequence of the uniqueness theorem:
\begin{coro}
	Let $y_0\in X_+$ and suppose that the hypothesis of Theorem \ref{uniqtheo} are satisfied.  
	Let $y$ be the unique solution of the silicosis system \eqref{silsys} such that $y(0)=y_0$. 
	If $(y^n)_{n\in\mathbb{N}}$ is the sequence of solutions of the truncated systems \eqref{siltrunc} 
	considered in the proof of Theorem \ref{existence}, then, as $n\to\infty$, $y^n(t)\to y(t)$ 
	uniformly on compact intervals of $[0,\infty)$.
\end{coro}
\begin{proof}
	In order to reach a contradiction, suppose that there is a subsequence of $(y^n)$, $(y^{n_j})$  say, for which there is 
	$T>0$ and $\eta>0$, such that, for all $j\in\mathbb{N}$, $\sup_{t\in[0,T]}|y^{n_j}(t)-y(t)|\geqslant \eta$. 
	We can take this subsequence in place of the whole approximating sequence in the proof of Theorem 
	\ref{existence} and conclude by that proof and the uniqueness theorem, that there is another subsequence 
	$(y^{n_{j'}})$ such that, for each $i\in\mathbb{N}$, as $j'\to\infty$, $y_i^{n_{j'}}(t)\to y_i(t)$, on $[0,T]$.
\end{proof}

\section{Semigroup property}

The last result we shall consider is a property that, together with the existence and the uniqueness results, is crucial for
the study of the dynamics of solutions: the semigroup property. 

Due to the fact that we could only
prove the uniqueness result stated in Theorem~\ref{uniqtheo} by imposing a more restrictive set of conditions on the
coefficients than those required for the proof of existence in Theorem~\ref{existence}, and, in particular, 
as written above in Corollary~\ref{nicecorol}, we can only proof uniqueness of solution in the full space $X_+$ 
when $(k_i)$ is a bounded sequence, 
we are going to prove the semigroup property for initial conditions $y_0$ satisfying an extra higher moment condition
like \eqref{higherinitialmoment}. In order to do this we need to start by proving that, under the assumptions
of Corollary~\ref{nicecorol}, if we take any initial condition $y_0 \in X_+^\mu := X^\mu \cap X_+,$ with $\mu \geqs 1$, where
\begin{equation}\label{normmu}
X^\mu :=  \Bigl\{y\in X \,: \|y\|_\mu:=|y_{1}| + \sum_{i=2}^\infty (i-1)^{\mu}|y_{i}|<\infty\Bigr\},
\end{equation}
then the solution stays in $X_+^\mu.$ Using this invariance property we then prove that, under those conditions,
the set of solutions in these subspaces form a $C_0$-semigroup.

\begin{prop}\label{invariance}
Let $(p_i),$ $(q_i),$ $(k_i)$  be nonnegative sequences satisfying $p_i=O(1)$, $q_i=O(i),$ and 
$k_i=O(i^\gamma)$ for some $\gamma \in [0,1].$ Take any $y_0 \in X_+^{1+\gamma}$ and let $y(\cdot)$ be the
unique solution of  \eqref{silsys} with $y(0)=y_0.$ Then $y(t)\in X_+^{1+\gamma}$ for all $t>0.$
\end{prop}

\begin{proof}
With the coefficients satisfying these assumptions and taking the sequence $(g_i)$
defined by $g_1=1$ and $g_i = (i-1)^{1+\gamma}$ if $i\geqs 2,$ 
all conditions of  Theorem~\ref{existence} (existence) and Theorem~\ref{uniqtheo} (uniqueness) are satisfied.
Truncating the initial condition and working with the solutions
to the $n_j$-truncated systems we can repeat the argument from \eqref{novaeq} until \eqref{giMiest}
with the $(g_i)$ above, concluding from \eqref{giMiest} that the solution $y$ of the infinite
dimensional system obtained by making $n_j\to\infty$ satisfies
$y_1(t) + \sum_{i=2}^\infty (i-1)^{1+\gamma}y_i(t) < C_1e^{C_2 t},$ for all $t\in [0, T]$ 
and with $C_1$ and $C_2$ the same as before.
Since under these conditions solutions are unique, 
through $y_0$ there is only one solution $y$ and
the argument above proves that $y(t)\in X_+^{1+\gamma}$, for all $t$.
\end{proof}

\begin{theorem}\label{semigroup}
Let $(p_i),$ $(q_i),$ $(k_i)$  be nonnegative sequences satisfying $p_i=O(1)$, $q_i=O(i),$ and 
$k_i=O(i^\gamma)$ for some $\gamma \in [0,1].$ Let $y_0\in X_+^{1+\gamma}.$
Denote by $T(\cdot)y_0$ the unique solution $y(\cdot\, ; y_0)$ of \eqref{silsys} satisfying the initial 
condition $y(0; y_0)=y_0.$
Then, $\{T(t): X_+^{1+\gamma} \to X_+^{1+\gamma} \left|\right.\, t\geqslant 0\}$ is a $C_0$-semigroup, i.e.,
\begin{enumerate}
\item[\text{(i)}] $T(0) = \text{id},$ the identity operator;
\item[\text{(ii)}] $T(t+s) = T(t)T(s),$ for all $t, s\geqslant 0;$
\item[\text{(iii)}] $(t,y_0)\mapsto T(t)y_0$ is a continuous mapping from 
$[0,\infty)\times X_+^{1+\gamma}$ into $X_+^{1+\gamma}.$
\end{enumerate}
\end{theorem}

\begin{proof}
Conditions (i) and (ii) are obvious from the fact that equations \eqref{M0t}--\eqref{xt} are autonomous.
To prove the continuity property in (iii) we first observe that,   due to the uniform convergence in $t\in [0, T]$ of 
$x^n(t)\to x(t)$ and $M_i^n(t)\to M_i(t)$ as $n\to\infty$,
 it is sufficient to prove the continuity separately in $t$ and in the initial condition $y_0$. 

Let us start by the continuity in $t$. 
The continuity of each map $y_i(\cdot\;;y_0),$ $i\in\mathbb{N},$ 
is a consequence of the definition of solution itself. For the continuity in the $X_+^{1+\gamma}$ norm we use  
the integrated version of the moments' equation 
\eqref{momentscompl}. In fact, by our hypothesis on the coefficients $k_i$, by taking $g_i=i^{1+\gamma},$ the hypothesis 
set (B) of Theorem \ref{momentsMi} is satisfied so that \eqref{momentscompl} holds for any $0\leqslant t_1<t_2$ and for any 
integer $m\geqslant 1.$ Fix some $T>0$, take $t_1=0,$ $t_2=T,$ and isolate the integral term of the left-hand side. 
Then, by using on the other terms, the uniform boundedness of $\|y(t;y_0)\|_{1+\gamma}$ for $t\in[0,T],$ and the 
dominated convergence theorem, we conclude that
$$
\phi_m(T,y_0):=\int_0^T\sum_{i=m}^\infty i^{1+\gamma}(p_i+q_i)M_i(s)\,ds\;\longrightarrow
\; 0,\qquad\text{ as }m\to\infty.
$$
Let $\varepsilon>0$ be given. Then, again by the uniform boundedness of $\|y(t;y_0)\|_{1+\gamma}$ for $t\in[0,T],$ 
there is $\delta_0>0,$ depending only on $T$ and $y_0$, such that the right-hand side of \eqref{momentscompl} lies in $[0,\varepsilon/3),$ 
if $0\leqslant t_1<t_2\leqslant T$ and $t_2-t_1<\delta_0.$ Fix the integer $m>2$ so that $\phi_m< \varepsilon/3.$ Then, 
there is $\delta_1\in (0,\delta_0),$ depending only on $T$ and $y_0,$ such that, if $0\leqslant t_1<t_2\leqslant T,$ 
and $t_2-t_1<\delta_1,$
$$
\begin{aligned}0\leqslant \int_{t_1}^{t_2}\sum_{i=1}^\infty i^{1+\gamma}(p_i+q_i)M_i(s)\,ds&\leqslant
\sum_{i=1}^{m-1} i^{1+\gamma}(p_i+q_i)\int_{t_1}^{t_2}M_i(s)\,ds+
\phi_m(T;y_0)\\
&< \frac{\varepsilon}{3}+\frac{\varepsilon}{3}\,,
\end{aligned}
$$
and hence, from  \eqref{momentscompl} we conclude that, 
$$
\left|\sum_{i=1}^\infty i^{1+\gamma}M_i(t_2)-\sum_{i=1}^\infty i^{1+\gamma}M_i(t_1)\right|<\varepsilon,
$$
thus, together with the continuity of $x(\cdot)$ and $M_0(\cdot),$ allowing us to conclude that,
for each $y_0$, $\|T(\cdot)y_0\|_{1+\gamma}=\|y(\cdot\;;y_0)\|_{1+\gamma}$ is continuous in $[0,\infty)$. 
This, together with the continuity of each coordinate $y_i(\cdot\;y_0)$ implies the continuity in the 
$X_+^{1+\gamma}$ norm as a result of standard weak$^*$ convergence results (see Lemma 3.3 in \cite{BCP}, for example).

It remains to prove in (iii) the continuity with respect to $y_0$ in $X^{1+\gamma}.$ 
Consider a sequence $\left(y_0^{n}\right)_{n\in\Nb_0}$
and assume that, for some $y_0\in X_+^{1+\gamma},$ we have $y_0^n\to y_0$ strongly in $X_+^{1+\gamma}$ as $n\to\infty$.
By our present assumptions, for each $y_0^n$ there is a unique globally defined solution
in $X_+^{\gamma}$, $y^n(\cdot):=T(\cdot)y_0^n$, and also
a unique global solution in $X_+^{\gamma}$, $y(\cdot):=T(\cdot)y_0.$ 
We now need to prove that, for each $t$,  $y^n(t)\to y(t)$ strongly in $X_+^{\gamma}$ as $n\to\infty$.

To this goal we start by repeating the proof of the existence result Theorem~\ref{existence} with $y^n$ taking the place of the 
solutions of the $n$-truncated systems. 
Actually, the only difference between what we do now and what was done in the proof of Theorem~\ref{existence} is
that now one must take $\chi_{i,j}\equiv 1$ and, instead of using Lemma~\ref{truncexist} for estimations of the moment
of $y^n$, we need to use Theorem~\ref{momentsMi}. This allows us to conclude that, for the above functions
$y^n=(x^n, M_0^n, M_1^n, \ldots)$, we have, when $n\to \infty,$ 
$y^n \stackrel{\ast}{\rightharpoonup} y$ in $X_+$ which, in particular, 
implies that for each $i$, $M_i^n(s)\to M_i(s)$ as $n\to \infty$. 

Now take a solution $y^n = (x^n, M_0^n, M_1^n, \ldots)$  and consider the 
moment equation  \eqref{momentscompl} satisfied by it with $m=1$, $t_1=0,$ $t_2=t$ and $g_i=i^{1+\gamma}.$ 
We want to prove that, when we pass to the limit $n\to\infty$, the same equation is valid for the limit solution $y=(x,M_0, M_1, \ldots).$
This clearly implies without additional effort that $\|y^n(t)\|_{1+\gamma} \to \|y(t)\|_{1+\gamma}$ which, together with the
corresponding version of Lemma 3.3 in \cite{BCP},  implies the result.
We shall make use of the bound \eqref{ynnorm} (which is valid for every solution) to get uniform bounds on the components $x^n(s)$ and $M_i^n(s)$
of the solution $y^n$, namely
\begin{equation}\label{boundxMy}
x^n(t)\;,\, M_i^n(t) \leqs \|y^n(t)\|\leqslant \|y^n_0\|+(\alpha+r)t\leqslant 2\|y_0\|+(\alpha+r)t
\end{equation}
where the last inequality holds for all sufficiently large $n$ (because $y_0^n\to y_0$ in $X_+^{1+\gamma}\subset X_+$).
The same holds for the components of the limit solution $y$.

Let $t\in [0, T]$ for any fixed $T<\infty,$ and write \eqref{momentscompl} as follows:
\begin{multline}\label{momentsemigroup}
		\sum_{i=1}^\infty i^{1+\gamma}M^n_i(t) \;=\; \sum_{i=1}^\infty i^{1+\gamma}M^n_i(0) - 
                     \int_{0}^{t}\sum_{i=1}^\infty i^{1+\gamma}(p_i+q_i)M^n_i(s)\,ds +\\
                    	+ \int_{0}^{t} k_{0}x^n(s)M^n_{0}(s)\,ds
		+\int_{0}^{t}x^n(s)\sum_{i=1}^{\infty} ((i+1)^{1+\gamma}-i^{1+\gamma})k_i  M^n_i(s)\,ds\,.
		\end{multline}

By the assumption on the initial data the first term in the right-hand side converges to 
$\sum_{i=1}^\infty i^{1+\gamma}M_i(0)$ as $n\to \infty.$

Using the bounds provided by \eqref{boundxMy},  the dominated convergence theorem applied to the second integral in the 
right-hand side of \eqref{momentsemigroup} allows for the convergence of that
integral to $ \int_{0}^{t}k_{0} x(s)M_{0}(s)\,ds$ as $n\to \infty.$

To deal with the third integral in the right-hand side of  \eqref{momentsemigroup} we take a fixed $m\in\Nb$ 
and separate the sum into a sum with $i\leqs m-1$ and another with $i\geqs m.$ For this ``tail sum'',  observe that, with the assumption 
that $k_i=O(i^\gamma)$,  applying Lagrange's  mean value theorem gets us
\[
 \bigl((i+1)^{1+\gamma}-i^{1+\gamma}\bigr)k_i \sim C_{\gamma}i^{2\gamma} = C_{\gamma}\frac {1}{i^{1-\gamma}}i^{1+\gamma}
\leqs \frac{C_{\gamma}}{m^{1-\gamma}}i^{1+\gamma}.
\]
Also, from \eqref{giMiest} and the fact that $C_1$ and $C_2$ in that bound can be chosen independent of $n$ 
(because $y_0^n\to y_0$ in $X_+^{1+\gamma}$) we have the estimate
\[
\sum_{i=m}^{\infty} ((i+1)^{1+\gamma}-i^{1+\gamma})k_i M^n_i(s) 
\leqs \frac{C_{\gamma}}{m^{1+\gamma}}\sum_{i=m}^{\infty}i^{1+\gamma}M^n_i(s)
\leqs  \frac{1}{m^{1+\gamma}}C_{\gamma}C_1e^{C_2s}
\]
and thus, by \eqref{boundxMy} and the dominated convergence theorem, we conclude that, for every $\varepsilon>0$ 
and $s\in [0, T]$, there is a $m_0$ such that, for all $m\geqs m_0$  the tail contribution
satisfies
\begin{equation}\label{tail}
\int_{0}^{t}x^n(s)\sum_{i=m}^{\infty} ((i+1)^{1+\gamma}-i^{1+\gamma})k_i  M^n_i(s)\,ds \leqs 
 \frac{1}{m^{1+\gamma}}C(T,\|y_0\|, \|y_0\|_{1+\gamma}) < \frac{1}{3}\varepsilon,
\end{equation}
where the constant $C(T,\|y_0\|, \|y_0\|_{1+\gamma})$ does not depend on $n.$
Clearly \eqref{tail} is also valid with $x^n$ and $M_i^n$ changed to $x$ and $M_i$, respectively.
For the sum with $1\leqs i\leqs m-1$ the pointwise convergence of $M_i^n(s)$ to $M_i(t)$ and the 
bound \eqref{boundxMy} allow us to apply again the dominated convergence theorem to conclude that, for
every $\varepsilon>0$, there exists a $n_0$ such that, for all $n>n_0$,
\begin{multline}\label{notail}
\left|
\int_0^t x^n(s)\sum_{i=1}^{m-1} ((i+1)^{1+\gamma}-i^{1+\gamma})k_i M^n_i(s)ds\; - \right. \\
\left. - \int_0^t x(s)\sum_{i=1}^{m-1} ((i+1)^{1+\gamma}-i^{1+\gamma})k_i M_i(s)ds
\right| < \frac{1}{3}\varepsilon
\end{multline} 
Plugging together \eqref{notail} and \eqref{tail} proves that the last integral in the right-hand side
of \eqref{momentsemigroup} converges to the integral with $x$ and $M_i$ in place of $x^n$ and $M_i^n$,
respectively.

We finally consider the first integral in the right-hand side of \eqref{momentsemigroup}.
Again we separate the sum into a part with $1\leqs i\leqs m-1$ and a tail contribution with $i\geqs m$.
For the tail term we can use \eqref{momentscompl}, Proposition~\ref{invariance},  and what we
proved above for the remaining integral terms to conclude that, for every $\varepsilon >0$
there exists $m$ and $n_0$ such that, for all $n>n_0$, 
\begin{equation}\label{beforelastbound}
 \int_{0}^{t}\sum_{i=m}^\infty i^{1+\gamma}(p_i+q_i)M^n_i(s)\,ds < \frac{1}{3}\varepsilon,
\end{equation}
and the same holds with $M_i^n$ substituted by $M_i$.
For the sum with $1\leqs i\leqs m-1$ it suffices to note that, for all $s\in [0, T]$ and
for $n>n_0$ (augmenting $n_0$ if needed)
\begin{equation}\label{lastbound}
\sum_{i=1}^{m-1} i^{1+\gamma}(p_i+q_i)M^n_i(s) \leqs (m-1)^{\gamma}
\max_{1\leqs i\leqs m-1}(p_i+q_i)(2\|y_0\|+(\alpha+r)T),
\end{equation}
and so, applying the dominated convergence theorem one last time, we conclude that, 
for every $\varepsilon >0$ there exists $m$ and $n_0$ such that, for all $n>n_0$, 
\begin{multline}\label{notail2}
\left|
\int_0^t \sum_{i=1}^{m-1} i^{1+\gamma}(p_i+q_i)M^n_i(s)ds 
 - \int_0^t \sum_{i=1}^{m-1} i^{1+\gamma}(p_i+q_i)M_i(s)ds
\right| < \frac{1}{3}\varepsilon.
\end{multline} 
Hence, by \eqref{beforelastbound} and \eqref{notail2} we conclude that, when $n\to\infty$,
the first integral in the right-hand side of \eqref{momentsemigroup} converges to a similar term with
$M_i$ in place of $M_i^n$.

We now have all the ingredients to pass to the limit $n\to\infty$  in both sides of \eqref{momentsemigroup} 
and obtain, for all $t\in [0, T],$
\begin{multline}\label{limitmomentsemigroup}
		\lim_{n\to\infty}\sum_{i=1}^\infty i^{1+\gamma}M^n_i(t) \;=\; \sum_{i=1}^\infty i^{1+\gamma}M_i(0) - 
                     \int_{0}^{t}\sum_{i=1}^\infty i^{1+\gamma}(p_i+q_i)M_i(s)\,ds +\\
                    	+ \int_{0}^{t} k_{0}x(s)M_{0}(s)\,ds
		+\int_{0}^{t}x(s)\sum_{i=1}^{\infty} ((i+1)^{1+\gamma}-i^{1+\gamma})k_i  M_i(s)\,ds\,.
		\end{multline}
Being $y=(x, M_0, M_1, \ldots)$ the unique solution of  \eqref{silsys} with initial condition $y_0\in X_+^{1+\gamma}$,
from Theorem~\ref{momentsMi} we conclude that the right-hand side of  \eqref{limitmomentsemigroup} is equal to
$\sum_{i=1}^\infty i^{1+\gamma}M_i(t)$. That is, given the arbitrariness of $T$,
\[
\lim_{n\to\infty}\sum_{i=1}^\infty i^{1+\gamma}M^n_i(t) = \sum_{i=1}^\infty i^{1+\gamma}M_i(t),\quad \forall t\geqs 0.
\]
This and the expression for the norm of $X^{1+\gamma}$ in \eqref{normmu} allow us to conclude without further effort that, 
with the present assumptions, solutions $y(\cdot\, ; y_0)$ of \eqref{silsys} depend on the
initial condition $y_0$ continuously in the norm topology of $X^{1+\gamma}$. 
This completes the proof. \end{proof}

\subsection*{Acknowlegments} We thank an anonymous referee for his/her helpful suggestions and
for   pointing out an error in the 
statement and proof of  the original version of Theorem~\ref{semigroup}.

%
%
\bibliographystyle{amsplain}

\end{document}